\documentclass[a4paper,11pt,twoside]{amsart}
\usepackage[utf8]{inputenc}
\usepackage{amssymb,amsmath,amsthm,amscd,mathtools}
\usepackage{enumerate}
\usepackage{verbatim}
\usepackage{lmodern}
\usepackage[dvipsnames]{xcolor}
\usepackage{hyperref}
\usepackage{scalerel,stackengine}

\stackMath

\newcommand{\C}{\ensuremath{\mathbb{C}}}
\newcommand{\R}{\ensuremath{\mathbb{R}}}
\newcommand{\N}{\ensuremath{\mathbb{N}}}
\newcommand{\Z}{\ensuremath{\mathbb{Z}}}

\renewcommand\epsilon\varepsilon

\newcommand{\Aut}{{\rm{Aut}}}

\newcommand{\norm}[1]{\left\lVert#1\right\rVert}

\newcommand\reallywidehat[1]{
	\savestack{\tmpbox}{\stretchto{\scaleto{\scalerel*[\widthof{\ensuremath{#1}}]{\kern-.6pt\bigwedge\kern-.6pt}
    {\rule[-\textheight/2]{1ex}{\textheight}}}{\textheight}}{0.5ex}}\stackon[1pt]{#1}{\tmpbox}}

\theoremstyle{definition}
\newtheorem{thmA}{Theorem}

\newtheorem{thm}{Theorem}[section]
\newtheorem{dfn}[thm]{Definition}
\newtheorem{lem}[thm]{Lemma}
\newtheorem{prp}[thm]{Proposition}
\newtheorem{cor}[thm]{Corollary}
\newtheorem{exm}[thm]{Example}
\newtheorem{rmk}[thm]{Remark}

\author{Dennis Heinig}
\address{\newline Westf\"alische Wilhelms-Universit\"at M\"unster, Mathematisches Institut
\newline Einsteinstra\ss{}e 62, 48149 M\"unster, Germany \newline \newline \normalfont{\texttt{mail.dheinig@gmail.com} \newline \texttt{tim.delaat@uni-muenster.de} \newline \texttt{timo.siebenand@uni-muenster.de}}}

\author{Tim de Laat}

\author{Timo Siebenand}

\title{Group $C^*$-algebras of locally compact groups acting on trees}

\date{}

\begin{document}

\begin{abstract}
We study the group $C^*$-algebras $C^*_{L^{p+}}(G)$ -- constructed from $L^p$-integrability properties of matrix coefficients of unitary representations -- of locally compact groups $G$ acting on (semi-)homogeneous trees of sufficiently large degree. These group $C^*$-algebras lie between the universal and the reduced group $C^*$-algebra. By directly investigating these $L^p$-integrability properties, we first show that for every non-compact, closed subgroup $G$ of the automorphism group $\mathrm{Aut}(T)$ of a (semi-)homogeneous tree $T$ that acts transitively on the boundary $\partial T$ and every $2 \leq q < p \leq \infty$, the canonical quotient map $C^*_{L^{p+}}(G) \twoheadrightarrow C^*_{L^{q+}}(G)$ is not injective. This reproves a result of Samei and Wiersma. We prove that under the additional assumptions that $G$ acts transitively on $T$ and that it has Tits' independence property, the group $C^*$-algebras $C^*_{L^{p+}}(G)$ are the only group $C^*$-algebras coming from $G$-invariant ideals in the Fourier-Stieltjes algebra $B(G)$. Additionally, we show that given a group $G$ as before, every group $C^*$-algebra $C^*_{\mu}(G)$ that is distinguishable (as a group $C^*$-algebra) from the universal group $C^*$-algebra of $G$ and whose dual space $C^*_\mu(G)^*$ is a $G$-invariant ideal in $B(G)$ is abstractly ${}^*$-isomorphic to the reduced group $C^*$-algebra of $G$.
\end{abstract}

\maketitle

\section{Introduction and statement of the main results} \label{sec:introduction}
There are several interesting ways to construct operator algebras from locally compact groups. Arguably the best known are the universal group $C^*$-algebra $C^*(G)$ and the reduced group $C^*$-algebra $C^*_r(G)$ of a locally compact group $G$. There is a canonical surjective ${}^*$-homomorphism $C^*(G)\twoheadrightarrow C^*_r(G)$, which is well known to be a ${}^*$-isomorphism if and only if $G$ is amenable. For non-amenable groups $G$, it is natural to investigate whether there are other interesting group $C^*$-algebras that lie between these two extremes.

For our purposes, a group $C^*$-algebra is a $C^*$-completion $A$ of $C_c(G)$ with surjective ${}^*$-homomorphisms from $C^*(G)$ to $A$ and from $A$ to $C^*_r(G)$ that extend the identity map on $C_c(G)$ (see Section \ref{subsec:constructinggroupcstaralgebras} for details):
\begin{equation*}
	C^*(G) \twoheadrightarrow A \twoheadrightarrow C^*_r(G).
\end{equation*}
If, moreover, both these canonical surjections are not injective, the algebra $A$ is said to be an exotic group $C^*$-algebra.

Apart from being interesting objects from an intrinsic point of view, group $C^{*}$-algebras of this type and related constructions, such as (exotic) crossed products, have received a lot of attention in recent years, in particular because of their relevance for the study of the Baum-Connes conjecture with coefficients and the strong Novikov conjecture (see e.g.~\cite{MR3514939}, \cite{MR3824785}, \cite{MR3837592}, \cite{antoninietal}).

A natural class of group $C^*$-algebras of a locally compact group $G$ comes from the $L^p$-integrability properties (for different $p$) of matrix coefficients of unitary representations. Let $p \in [1, \infty]$. A unitary representation $\pi \colon G \to \mathcal{U}(\mathcal{H})$ is called an $L^{p+}$-representation if for every $\varepsilon > 0$, sufficiently many of its matrix coefficients are elements of $L^{p+\varepsilon}(G)$. Given $p \in [2,\infty]$, completing the algebra $C_c(G)$ with respect to the natural norm coming from the collection of $L^{p+}$-representations of $G$ yields a group $C^*$-algebra of $G$, denoted by $C^*_{L^{p+}}(G)$ (see Section \ref{subsec:cstarlp} for the precise construction).

In this article, we study the group $C^*$-algebras $C^*_{L^{p+}}(G)$ for certain classes of (non-discrete) totally disconnected locally compact groups $G$ acting on trees. Our starting point is the following result, which shows that given an appropriate locally compact group $G$ acting  on a semi-homogeneous tree  of sufficiently large degree, the group $C^*$-algebras $C^*_{L^{p+}}(G)$ are pairwise distinguishable for $p \in [2,\infty]$.
\begin{thmA} \label{thm:groupcstaralgebrastrees}
Let $T$ be a semi-homogeneous tree of degree $(d_0,d_1)$ with $d_0,d_1\geq 2$ and $d_0+d_1\geq 5$, and let $G$ be a non-compact, closed subgroup of the automorphism group $\mathrm{Aut}(T)$. Suppose that $G$ acts transitively on the boundary $\partial T$. For $2 \leq q < p \leq \infty$, the canonical quotient map
	\[
		C^*_{L^{p+}}(G) \twoheadrightarrow C^*_{L^{q+}}(G)
	\]
is not injective.
\end{thmA}
We give a proof of this result in Section \ref{sec:sphericalsubcstaralgebras}. It reproves a result of Samei and Wiersma (see \cite[Proposition 4.11 and Example 5.9]{sameiwiersma2}). Our approach is similar to theirs, in the sense that it relies on establishing the integrable Haagerup property for the groups under consideration, which together with the Kunze-Stein property (which is known for these groups) implies the theorem. However, our approach towards the integrable Haagerup property strongly relies on harmonic analysis and representation theory rather on geometric considerations.

Group $C^*$-algebras constructed from $L^p$-integrability properties of matrix coefficients have already been investigated extensively for discrete groups and for Lie groups. The systematic study of such algebras (in the setting of discrete groups) was initiated in \cite{MR3138486}. This lead to an analogue of Theorem \ref{thm:groupcstaralgebrastrees} for (non-amenable) discrete groups containing a non-abelian free subgroup \cite{MR3238088}, \cite{MR3705441}. In the setting of Lie groups, Wiersma proved an analogue of Theorem \ref{thm:groupcstaralgebrastrees} for $\mathrm{SL}(2,\mathbb{R})$ \cite{MR3418075}, which was generalised to the Lie groups $\mathrm{SO}_0(n,1)$ and $\mathrm{SU}(n,1)$ in \cite{sameiwiersma2}. In \cite{delaatsiebenand1}, the second-named and the third-named author generalised this to all classical simple Lie groups with real rank one, including the ones with property (T), which could not be dealt with before.

A related question is whether the algebras $C^*_{L^{p+}}(G)$ are the only group $C^*$-algebras of the groups considered in Theorem \ref{thm:groupcstaralgebrastrees}. A positive answer is too much to hope for, but under two additional assumptions on the groups, we can show that the algebras $C^*_{L^{p+}}(G)$ are the only ones that can be constructed from $G$-invariant ideals in the Fourier-Stieltjes algebra $B(G)$ of $G$ (see Section \ref{subsec:constructinggroupcstaralgebras} for details).
\begin{thmA} \label{thm:ideals}
Let $T$ be a homogeneous tree of degree $d \geq 3$, and let $G$ be a non-compact, closed subgroup of the automorphism group $\mathrm{Aut}(T)$. Suppose that $G$ acts transitively on the vertices of $T$ and on the boundary $\partial T$ and that $G$ satisfies Tits' independence property. If $C^*_{\mu}(G)$ is a group $C^*$-algebra of $G$ such that its dual space $C^*_{\mu}(G)^*$ is a $G$-invariant ideal of $B(G)$, then there exists a unique $p \in [2,\infty]$ such that $B_{L^{p+}}(G)=C^*_{\mu}(G)^*$, where $B_{L^{p+}}(G):=C^*_{L^{p+}}(G)^*$.
\end{thmA}
We also prove this theorem for the totally disconnected group $\mathrm{SL}(2,\mathbb{Q}_p)$, which does not have Tits' independence property.

Group $C^*$-algebras constructed from $G$-invariant ideals in $B(G)$ play an important role in the theory of exotic crossed product functors due to their good behaviour with respect to important invariants, such as K-theory (see \cite{MR3824785}, \cite{MR3837592}). This behaviour plays an important role in the proof of Theorem \ref{thm:ideals}, which is given in Section \ref{sec:sphericalsubcstaralgebras}. The strategy of this proof was first used by the third-named author in \cite{siebenand}, where he showed an analogue of Theorem \ref{thm:ideals} for $\mathrm{SL}(2,\mathbb{R})$ and $\mathrm{SL}(2,\mathbb{C})$. The analogue for $\mathrm{SL}(2,\mathbb{R})$ had already been covered before by representation-theoretic methods in \cite{MR3418075}.

Furthermore, similar methods lead to the following result, the proof of which is also given in Section \ref{sec:sphericalsubcstaralgebras}. This result indicates that canonical (non)-${}^*$-isomorphism of group $C^*$-algebras is very subtle.
\begin{thmA}\label{thm:star_iso}
Let $T$ be a semi-homogeneous tree of degree $(d_0,d_1)$ with $d_0,d_1\geq 2$ and $d_0+d_1\geq 5$, and let $G$ be a non-compact, closed subgroup of the automorphism group $\mathrm{Aut}(T)$. Suppose that $G$ acts transitively on the boundary $\partial T$ and that $G$ satisfies Tits' independence property. Then every group $C^*$-algebra $C^*_\mu(G)$ of $G$ that is distinguishable (as a group $C^*$-algebra) from the universal group $C^*$-algebra of $G$ and whose dual space $C^*_\mu(G)^*$ is a $G$-invariant ideal of $B(G)$ is (abstractly) ${}^*$-isomorphic to the reduced group $C^*$-algebra of $G$.
\end{thmA}
From the work of Bruhat and Tits \cite{MR0327923}, it is known that semi-simple algebraic rank one groups resemble the structure of the groups considered above. Indeed, for every reductive group over a local field, Bruhat and Tits constructed a geometric object -- nowadays called Bruhat-Tits building -- on which the group admits a natural action. These buildings can be viewed as a generalisation of Riemann symmetric spaces, and in the case of groups of rank one, the Bruhat-Tits building is a (semi-)homogeneous tree. In this way, our results may be applied to (appropriate classes of) rank one algebraic groups. In particular, it is known that the action of a simple algebraic group of rank one over a non-Archimedean local field on the boundary of its Bruhat-Tits tree is transitive, so Theorem \ref{thm:groupcstaralgebrastrees} directly applies.

\section*{Acknowledgements}
\noindent We thank Siegfried Echterhoff for interesting discussions and several useful comments, Marc Burger for pointing out \cite{sallyIntroductionAdicFields1998} to us, and Sven Raum for useful remarks on an earlier version of this article.

TdL and TS are supported by the Deutsche Forschungsgemeinschaft - Project-ID 427320536 - SFB 1442, as well as under Germany's Excellence Strategy - EXC 2044 -  390685587, Mathematics Münster: Dynamics - Geometry - Structure.

\section{Group $C^*$-algebras} \label{subsec:unitarydual}
We now recall some basic theory of group $C^*$-algebras. In this section, let $G$ be a locally compact group, equipped with a fixed Haar measure $\mu_G$.

\subsection{Weak containment and the unitary dual} \label{subsec:unitarydual}
A matrix coefficient of a unitary representation $\pi \colon G \to \mathcal{U}(\mathcal{H})$ is an (automatically bounded and continuous) function of the form $\pi_{\xi,\eta} \colon s \mapsto \langle \pi(s)\xi,\eta \rangle$, where $\xi,\eta \in \mathcal{H}$. A matrix coefficient $\pi_{\xi,\eta}$ is called diagonal if $\xi=\eta$, i.e.~if it is of the form $\pi_{\xi,\xi}$ for some $\xi \in \mathcal{H}$.

Let $\pi_1$ and $\pi_2$ be unitary representations of $G$. If every diagonal matrix coefficient of $\pi_1$ can be approximated uniformly on compact subsets of $G$ by finite sums of diagonal matrix coefficients of $\pi_2$, then the representation $\pi_1$ is said to be weakly contained in $\pi_2$.

For a locally compact group $G$, let $\widehat{G}$ denote its unitary dual, i.e.~the set of (unitary) equivalence classes of irreducible unitary representations equipped with the Fell topology. For a subset $S$ of $\widehat{G}$, the closure $\overline{S}$ of $S$ in the Fell topology consists of the elements of $\widehat{G}$ which are weakly contained in $S$. The subspace of $\widehat{G}$ consisting of all elements from $\widehat{G}$ that are weakly contained in the left regular representation $\lambda \colon G \to \mathcal{U}(L^2(G))$ is called the reduced unitary dual and denoted by $\widehat{G}_r$. For details, we refer to \cite{MR0458185}.

\subsection{Constructing group $C^{*}$-algebras} \label{subsec:constructinggroupcstaralgebras}
A group $C^*$-algebra associated with $G$ is a $C^{*}$-completion $C^*_\mu(G)$ of $C_c(G)$ with respect to a $C^*$-norm $\|.\|_{\mu}$ satisfying $\|f\|_u \geq \|f\|_{\mu} \geq \|f\|_r$ for all $f \in C_c(G)$, where $\|.\|_u$ and $\|.\|_r$ are the universal and the reduced $C^*$-norm, respectively. The identity map from $C_c(G)$ to $C_c(G)$ induces canonical surjective ${}^*$-homomorphisms $C^{*}(G) \twoheadrightarrow C^*_\mu(G)$ and $C^*_\mu(G) \twoheadrightarrow C^{*}_r(G)$. If both the quotient map $C^{*}(G) \twoheadrightarrow C^*_\mu(G)$ and the quotient map $C^*_\mu(G) \twoheadrightarrow C^{*}_r(G)$ are non-injective, then the algebra $C^*_\mu(G)$ is called an exotic group $C^*$-algebra. More generally, two group $C^*$-algebras
$C^*_{\mu_1}(G)$ and $C^*_{\mu_2}(G)$ are said to be distinguishable
if the corresponding $C^*$-norms $\|\cdot\|_{\mu_1}$ and $\|\cdot\|_{\mu_2}$
on $C_c(G)$ differ.

One way to construct group $C^*$-algebras, which goes back to \cite{MR3514939}, is by defining a $C^*$-norm that naturally comes from an appropriate subset of the unitary dual. More precisely, if $\widehat{G}$ and $\widehat{G}_r$ are the unitary and the reduced unitary dual (see Section \ref{subsec:unitarydual}), respectively, a subset $S \subset \widehat{G}$ is said to be admissible if $\widehat{G}_r \subset \overline{S}$. For such an admissible $S \subset \widehat{G}$, we can define a $C^*$-norm on $C_c(G)$ by
\[
	\|f\|_S:=\sup\{\|\pi(f)\| \mid \pi \in S\}.
\]
The corresponding completion $C^*_S(G)$ is a group $C^{*}$-algebra.
\begin{dfn} \label{dfn:ideal}
Let $G$ be a locally compact group. An ideal of $\widehat{G}$ is a subset $S \subset \widehat{G}$ such that for every $\pi \in S$ and every unitary representation $\rho$ of $G$, the unitary representation $\pi \otimes \rho$ is weakly contained in $S$.
\end{dfn}
Note that non-empty ideals are automatically admissible. Taking $S$ to be an ideal in the above construction has certain analytic advantages, as will be explained below.

In \cite{MR3141810}, another construction of group $C^*$-algebras was described. Recall that the Fourier-Stieltjes algebra $B(G)$ of a locally compact group $G$ is the Banach algebra consisting of matrix coefficients of unitary representations of $G$. The Fourier-Stieltjes algebra $B(G)$ can be identified canonically with the dual space $C^*(G)^*$ of $C^*(G)$ through the pairing given by $\langle \varphi, \, f \rangle = \int \varphi f \mathrm{d}\mu_G$, with $\varphi \in B(G)$ and $f\in C_c(G) \subset C^*(G)$. Let $B_r(G) \subset B(G)$ be the dual space of the reduced group $C^*$-algebra $C^*_r(G)$. If $E \subset B(G)$ is a weak*-closed $G$-invariant subspace of $B(G)$ that contains $B_r(G)$, then
\begin{align*}
	C_E^*(G) = C^*(G)/{}^{\perp}E
\end{align*}
is a group $C^*$-algebra. Here ${}^\perp E = \{ x\in C^*(G)\mid \langle\varphi, x\rangle = 0 \; \forall \varphi \in E\}$ denotes the pre-annihilator of $E$.

The two constructions recalled above are closely related: If $C^*_\mu(G)$ is a group $C^*$-algebra of $G$, then $\widehat{C^*_\mu(G)}\subset \widehat{G}$ is a closed ideal in $\widehat{G}$ if and only if the dual space $C^*_\mu(G)^*$ of $C^*_\mu(G)$ is a $G$-invariant ideal in $B(G)$. An explicit proof of this fact (which is well known to experts) can be found in \cite[Proposition 2.2]{delaatsiebenand1}.

\subsection{$L^p$-integrability of matrix coefficients and group $C^*$-algebras} \label{subsec:cstarlp}
We now consider unitary representations with certain $L^p$-integrability conditions on their matrix coefficients.
\begin{dfn} \label{dfn:lpplusrepresentation}
	Let $\pi \colon G \to \mathcal{U}(\mathcal{H})$ be a unitary representation, and let $p \in [1, \infty]$.
\begin{enumerate}[(i)]
	\item The representation $\pi$ is an $L^{p}$-representation if there exists a dense subspace 
		$\mathcal{H}_0 \subset \mathcal{H}$ such that for all $\xi,\eta \in \mathcal{H}_0$, we have $\pi_{\xi,\eta}\in L^p(G)$.
	\item The representation $\pi$ is an $L^{p+}$-representation if for all $p'\in (p,\infty]$, it is an $L^{p'}$-representation.
\end{enumerate}
\end{dfn}
These (and similar) notions have been studied extensively in the area of harmonic analysis on Lie groups.

We now consider the group $C^*$-algebras associated with these classes of representations. Note that in general, we cannot just take $S$ to be the set of (equivalence classes of) $L^{p+}$-representations and use the first construction above. Indeed, the set $S$ can be empty, e.g.~for non-compact locally compact abelian groups. Therefore, we define the $C^*$-norm in terms of unitary representations that are not necessarily irreducible.

Let $G$ be a locally compact group and $p\in[2,\infty]$. Let $C^*_{L^p}(G)$ and $C^*_{L^{p+}}(G)$ denote the group $C^*$-algebras obtained as the completions $C_c(G)$ with respect to norms
\begin{align*}
	\|\cdot \|_{L^p} &\colon C_c(G)\to [0,\infty),\, f \mapsto 
	\sup \lbrace \|\pi(f)\| \mid \pi \mbox{ is a } L^p\mbox{-representation}\} \mbox { and}\\
	\|\cdot \|_{L^{p+}}& \colon C_c(G)\to [0,\infty),\, f \mapsto 
	\sup \{ \|\pi(f)\| \mid \pi \mbox{ is a } L^{p+}\mbox{-representation}\},
\end{align*}
respectively. This essentially goes back to \cite{MR3138486}, where the algebras $C^*_{L^p}(G)$ were constructed for discrete groups $G$.

It is known that whenever $p\in [2,\infty]$, the dual spaces $\widehat{C^*_{L^{p}}(G)}$ and $\widehat{C^*_{L^{p+}}(G)}$ of $C^*_{L^{p}}(G)$ and $C^*_{L^{p+}}(G)$, respectively, are ideals in $\widehat{G}$ (in the sense of Definition \ref{dfn:ideal}).

Matrix coefficients (being bounded and continuous functions) that are in $L^p(G)$ for some $p \in [1,\infty]$ are automatically in $L^r(G)$ for all $r \geq p$. It follows that whenever $q \leq p$, we have $\norm{\cdot}_{L^{q+}}\geq \norm{\cdot}_{L^{p+}}$. Hence, the identity map on $C_c(G)$ extends to a canonical surjective ${}^*$-homomorphism $C^*_{L^{p+}}(G) \twoheadrightarrow C^*_{L^{q+}}(G)$. In order to find distinguishable group $C^*$-algebras, the relevant question is, under which conditions this map is not injective.

\subsection{Kunze-Stein groups and $L^{p+}$-group-$C^*$-algebras}
The distinguishability of the group $C^*$-algebras $C^*_{L^{p+}}(G)$ is especially well understood for Kunze-Stein groups. The investigation of the $L^{p+}$-group-$C^*$-algebras of such groups goes back to Samei and Wiersma \cite{sameiwiersma2}.

Recall that a locally compact group $G$ is called a Kunze-Stein group if
the convolution product on $C_c(G)$ extends to a bounded bilinear 
map
\begin{equation*}
	L^p(G)\times L^2(G)\to L^2(G)
\end{equation*}
for every $p\in [1,2)$.

Let us recall an important result on the algebras $C^*_{L^{p+}}(G)$ of Kunze-Stein groups $G$ (see {\cite[Theorem 5.3]{sameiwiersma2}}).
\begin{thm} \label{KS:cor_dual_Lpplus_algebra}
	Let $G$ be a Kunze-Stein group and $p\in [2,\infty]$. Then
		\begin{align*}
			C^*_{L^{p+}}(G)^* \subset L^{p+\varepsilon}(G) \textrm{ for all } \varepsilon > 0.
		\end{align*}
\end{thm}

\subsection{K-amenability and group $C^*$-algebras}
One of the key tools in our proofs is that group $C^*$-algebras whose dual space is an ideal in the Fourier-Stieltjes algebra behave particular well under K-theory if the group $G$ is K-amenable. Recall that a locally compact group $G$ is said to be K-amenable if the unit element $1_G$ of the Kasparov ring KK$^G(\mathbb{C},\mathbb{C})$ can be represented by a $\mathbb{C}$-$\mathbb{C}$-Kasparov $G$-module $(X,\gamma,\phi,F)$ such that $\gamma$ is weakly contained in the left regular representation of $G$ (see \cite{MR0716254}, \cite{MR0757995}).

The definition of K-amenability is slightly technical. By \cite[Theorem 1.3]{MR0757995}, the groups under consideration in this article are K-amenable. (More generally, this is true for all second countable locally compact groups with the Haagerup property \cite{MR1703305}). Hence, these technicalities do not play an important role in our arguments.

Specifically, we use the following result (see {\cite{MR3837592}, \cite{MR3824785}).
\begin{thm} \label{CR:thm:KK_eq_abs_grp_alg}
	Let $G$ be a $K$-amenable, second countable, locally compact group. If $C^*_\mu(G)$ is a group $C^*$-algebra of $G$ such that
	$C^*_\mu(G)^*$ is an ideal in $B(G)$, then the canonical quotient maps
	$q \colon C^*(G)\to C^*_\mu(G)$ and $s \colon C^*_\mu(G)\to C^*_r(G)$ are KK-equivalences, i.e.~$[q] \in \mathrm{KK}(C^*(G),C^*_{\mu}(G))$ and $[s] \in \mathrm{KK}(C^*_{\mu}(G),C^*_r(G))$ are invertible.
\end{thm}

\subsection{Group $C^*$-algebras and Gelfand pairs}
Let $G$ be a locally compact group and $K$ a compact subgroup of $G$. A function $\varphi \colon G \to \mathbb{C}$ is said to be $K$-bi-invariant if $\varphi(k_1sk_2)=\varphi(s)$ for all $s \in G$ and $k_1,k_2 \in K$. The pair $(G,K)$ is called a Gelfand pair if the ${}^*$-subalgebra $C_c(K \backslash G / K)$ of $C_c(G)$ consisting of all $K$-bi-invariant elements of $C_c(G)$ is commutative. 

Given a Gelfand pair $(G,K)$, a non-trivial $K$-bi-invariant Radon measure $\chi \colon C_c(G) \to \C$ on $G$ is called spherical if it restricts to an algebra
homomorphism on $C_c(K\backslash G/K)$. Each spherical Radon measure $\chi$ on $G$ is absolutely continuous with respect to the Haar measure $\mu_G$, and there is a $K$-bi-invariant, continuous function $\varphi\in C(G)$ with $\varphi(e) = 1$ such that
\begin{align*}
	\chi(f) = \int f(s)\varphi(s^{-1})\mathrm{d}\mu_G(s)
\end{align*}
for all $f\in C_c(G)$. A $K$-bi-invariant continuous function
	$\varphi\in C(G)$ with $\varphi(e) = 1$ such that the map $C_c(K\backslash G/K)\to \C, \, f \mapsto \int f(s)\varphi(s^{-1})\mathrm{d}\mu_G(s)$ forms an algebra homomorphism is called a spherical function for $(G,K)$.
	
An irreducible unitary group representation $\pi \colon G \to \mathcal{U}(\mathcal{H})$ is called spherical (or class one) for the Gelfand pair $(G, K)$ if the vector space $\mathcal{H}^K$ of $K$-invariant vectors is one-dimensional. We write $(\widehat{G}_K)_1$ for the (equivalence classes of) spherical representations of $G$. The space $(\widehat{G}_K)_1$ is called the spherical unitary dual of $(G,K)$. For details on Gelfand pairs and spherical functions, we refer to \cite{MR2328043}.

Recall that the universal group $C^*$-algebra $C^*(G)$ of $G$ comes together with a canonical unitary group representation $\iota_G \colon G \to \mathcal{UM}(C^*(G))$ (where $\mathcal{M}(C^*(G))$ denotes the multiplier algebra of $C^*(G)$) given by
\[
	(\iota_G(s)f)(t)=f(s^{-1}t)
\]
for $f \in C_c(G)$ and $s,t \in G$. The representation $\iota_G$ is called the universal group representation.

Let $G$ be a locally compact group, and let $K$ be a compact subgroup of $G$. Now, let $\mu_K$ be the normalized Haar measure on $K$. By $p_K \in \mathcal{M}(C^*(G))$, let us denote the orthogonal projection defined by
	\begin{align*}
		p_K x = \int \iota_G(k)x\, \mathrm{d}\mu_K(k)
	\end{align*}
	for $x\in C^*(G)$. Moreover, for a general group $C^*$-algebra $C^*_\mu(G)$
	of $G$, we denote by $p_{K,\mu}\in \mathcal{M}(C^*_\mu(G))$ the orthogonal
	projection $\overline{q}(p_K)$, where $q\colon C^*(G)\to C^*_\mu(G)$ is the canonical quotient map and $\overline{q}$ denotes the unique extension of $q$ to a ${}^*$-homomorphism on $\mathcal{M}(C^*(G))$.
	
	We write $C_\mu^*(K\backslash G/K)$ for the completion of $C_c(K\backslash G/K)$ in $C^*_\mu(G)$. It is a commutative sub-C*-algebra of $C^*_\mu(G)$ whenever $(G,K)$ is a Gelfand pair. Moreover, the following holds.
	
\begin{prp}\label{GA:gelfand:abelian_proj}
For every group $C^*$-algebra $C_\mu^*(G)$ of $G$, we have
\[
	C^*_\mu(K\backslash G/K) = p_{K,\mu}C_\mu^*(G)p_{K,\mu}.
\]
In particular, $p_{K,\mu}$ is an abelian projection if $(G,K)$ is a Gelfand pair.
\end{prp}
\begin{proof}
		Let $q \colon C^*(G)\to C^*_\mu(G)$ be the canonical
		quotient map, and let $\iota_{G,\mu} := \overline{q}\circ\iota_G$, where
		$\overline{q} \colon \mathcal{M}(C^*(G))\to \mathcal{M}(C^*_\mu(G))$ is as above.
		For $x\in C_\mu^*(G)$, let $x'\in C^*(G)$ be an element satisfying $q(x') = x$.
		We have
		\begin{align*}
			p_{K,\mu}x = q(p_K x')= q\left( \int \iota_G(k)x'\, \mathrm{d}\mu_K(k)\right) = \int \iota_{G,\mu}(k)x\, \mathrm{d}\mu_K(k).
		\end{align*}
		Furthermore, since $q$ is the identity on $C_c(G)$, for all $f\in C_c(G)$, 
		we have $\iota_{G,\mu}(k)f(t) = f(k^{-1}t)$ and $f\iota_{G,\mu}(k) (t) = f(tk)$
		for all $k\in K$ and $t\in G$. Hence, the elements
		\[
		\int \iota_{G,\mu}(k)f\, \mathrm{d}\mu_K(k)  = p_{K,\mu}f
		\]
		and 
		\[
		\int f\iota_{G,\mu}(k)\,\mathrm{d}\mu_K(k) = fp_{K,\mu}
		\]
		belong to $C_c(G)$. Now the $K$-bi-invariance of $\mu_K$ implies that $p_{K,\mu}fp_{K,\mu}$ lies in $C_c(K\backslash G/K)$. 
		In other words, we have
		\[
			p_{K,\mu}C_c(G)p_{K,\mu}\subset C_c(K\backslash G/K).
		\]
		The argument above also shows that for a function $f\in C_c(K\backslash G/K)$, we have
		$p_{K,\mu}fp_{K,\mu} = f$. This implies
		$C_c(K\backslash G/K)\subset p_{G,\mu}C_c(G)p_{G,\mu}$.
		
		The compression $E_{K,\mu}\colon C^*_\mu(G)\to C^*_\mu(G),\,
		x\mapsto p_{K,\mu}xp_{K,\mu}$ of the identity map on $C_\mu^*(G)$
		by $p_{K,\mu}$ is contractive. Hence, we have
		$p_{K,\mu}C^*_\mu(G)p_{K,\mu}= E_{K,\mu}(C^*_\mu(G))
		\subset C^*_\mu(K\backslash G/K)$. We also have
		$C^*_\mu(K\backslash G/K)\subset p_{K,\mu}C^*_\mu(G)p_{K,\mu}$,
		which completes the proof.
	\end{proof}
\begin{rmk}\label{GA:rem:GP_spherical_ideal_morita_eq}
	Let $C^*_\mu(G)$ be a group $C^*$-algebra of $G$. We denote by
	$C^*_\mu(G,K)$ the closed ideal in $C^*_\mu(G)$
	generated by the projection $p_{K,\mu}$ in $\mathcal{M}(C^*_\mu(G))$ and
	call it the spherical ideal of $C^*_{\mu}(G)$ for the Gelfand pair
	$(G,K)$. The left ideal $C^*_\mu(G)p_{K,\mu}$ and the right ideal $p_{K,\mu} C_\mu^*(G)$ form an imprimitivity $C^*_\mu( G,K)$-$C^*_\mu(K\backslash G/K)$-bi\-mo\-dule and $C^*_\mu(K\backslash G/K)$-$C^*_\mu( G,K)$-bi\-mo\-dule, respectively. They extend, in a canonical way, to a partial imprimitivity $C_\mu^*(G)$-$C^*_\mu(K\backslash G/K)$-bimodule $C^*_\mu(G/K)$ and $C^*_\mu(K\backslash G/K)$-$C_\mu^*(G)$-bi\-mo\-dule $C^*_\mu(K\backslash G)$, respectively.
		 
		 Note that the spectrum $\reallywidehat{C^*_\mu( G,K)}$ of
		 $C^*_\mu( G,K)$ can be
		 be identitfied with the open subset
		 $\lbrace [\pi]\in \widehat{C^*_\mu(G)}\mid \pi(C^*_\mu( G,K)) \neq 0 \rbrace$ of
		 $\widehat{C^*_\mu(G)}$ via
		 the topological embedding
		 \begin{align*}
		 	\reallywidehat{C^*_\mu( G,K)}\to \widehat{C_\mu^*(G)},\,
		 	[\pi]\mapsto [\overline{\pi}],
		 \end{align*}
		 where $\overline{\pi}\colon C_\mu^*(G)\to \mathcal{B}(\mathcal{H})$ is the
		 unique extension of $\pi\colon C^*_\mu( G,K) \to \mathcal{B}(\mathcal{H})$.
	\end{rmk}

\section{Trees and their automorphism groups}
\subsection{Trees and their boundaries}
A tree $T$ is an undirected, connected, acyclic graph (without loops and without multiple edges). We write $V(T)$ for the vertex set and $E(T) \subset \lbrace \lbrace x_0,x_1 \rbrace
		\subset V(T)\mid \vert \lbrace x_0,x_1 \rbrace \vert = 2\rbrace$ for the set of
	edges. The degree $\mathrm{deg}_T(x)$ of each
	vertex $x\in V(T)$ is the number of all edges containing
	$x$. A tree $T$ is locally finite if the degree of every vertex of $T$ is finite. 
	A particularly important class of locally finite trees is the class of semi-homogeneous trees. A tree is called	semi-homogeneous of degree $(d_0,d_1)\in \N^2$ if each vertex is of degree
	$d_0$ or $d_1$ and if for each edge $e$ of $T$, we have
	$\lbrace \mathrm{deg}_T(x) \mid x\in e \rbrace = \lbrace d_0,d_1 \rbrace$.
	Furthermore, a semi-homogeneous tree of degree $(d_0,d_1)$ is said
	to be homogeneous of degree $d_0$ if $d_0 = d_1$.
	
	Given a tree $T$, its vertex set $V(T)$ admits a canonical metric, turning $V(T)$ into a metric space. Indeed, recall that a path $c$ from $x\in V(T)$ to $y\in V(T)$
	is a finite sequence $c\colon \lbrace 0,\dots,n\rbrace \to V(T)$, with $n\in \N_0$, such that
	$c(0) = x$, $c(n) = y$ and $\lbrace c(i-1), c(i)\rbrace \in E(T)$ for all
	$1\leq i \leq n$. For all $x,y \in V(T)$, there is a unique injective path from
	$x$ to $y$. Such a path is called a geodesic and its range is called a
	geodesic segment and is denoted by $[x,y]$. The metric $d_c\colon V(T)\times V(T)\to [0,\infty)$ on $V(T)$, also called the shortest-path metric,
	is defined by
	\begin{align*}
		d_c(x,y) = \vert\, [x,y]\,\vert -1
	\end{align*}
	for $x,y\in V(T)$,
	where $\vert \, [x,y]\, \vert$ denotes the cardinality of the set $[x,y]$.
	
	The metric space $(V(T),d_c)$ has a canonical compactification, which we also recall at this point. An isometry $c\colon \N_0\to V(T)$, where $\N_0$ is equipped with the canonical distance, is called a ray or an infinite chain. Two rays $c$ and $c'$ in $T$ are called cofinal (denoted $\sim_{\textrm{cofin}}$) if there exists an $l_0 \in \N_0$ and an $m \in \Z$ such that for all $l \geq l_0$, we have $c({l+m})=c'(l)$. Cofinality defines an equivalence relation on the set of rays.

The boundary $\partial T$ of a tree $T$ is defined as the set of equivalence classes of rays with respect to the relation of cofinality:
\[
		\partial T:=\{c\colon\N_0 \to V(T) \mid \alpha \textrm{ is a ray in $T$ }\}/\sim_{\textrm{cofin}}.
\]
After fixing a vertex $x$ in $T$, every boundary point $\omega\in\partial T$ can be represented by the unique chain starting at $x$ in the equivalence class $\omega$. The range of this chain will be denoted $[x,\omega)$, and $\omega_i(x)$ will denote the $i$-th vertex of $[x,\omega)$, starting with $\omega_0(x)=x$.

The shortest-path metric induces the discrete topology on $T$. The set $T \cup \partial T$ carries a natural compact topology, with respect to which $T$ is dense in $T \cup \partial T$. Indeed, with every $\omega\in\partial T$, we associate the neighbourhood basis $\left\{E_x(y) \mid y\in[x,\omega)\right\}$, where $E_x(y)$ consists of all vertices and endpoints of infinite chains including $y$ but no other vertex from $[x,y]$, and we let $\tau$ denote the topology on $T\cup\partial T$ induced by these neighbourhood bases. Setting $\Omega_x(y)=E_x(y)\cap\partial T$, the sets $\Omega_x(y)$, with $y \in [x,\omega)$, form a neighbourhood basis of $\omega\in\partial T$ with respect to the relative topology on $\partial T$. These topologies do not depend on the choice of the vertex $x$. For $x \in V(T)$ and $n \in \N$, the set $\left\{\Omega_x(y) \mid d(x,y)=n\right\}$ is a partition of $\partial T$ into compact open sets.

For details on trees and their boundaries, we refer to \cite{MR1152801}.

\subsection{Automorphism groups of trees}
Let $T$ be a locally finite tree, and let $\mathrm{Aut}(T)$ denote its automorphism group. With respect to the shortest-path metric, every automorphism of $T$ is an isometry of $(V(T),d_c)$ and vice versa. As an isometry group of a proper metric space, $\mathrm{Aut}(T)$ is, equipped with the compact-open topology, a second countable, locally compact group. In fact, every closed subgroup $G$ of $\mathrm{Aut}(T)$ is a locally compact, totally disconnected group. Indeed, for every finite, complete subtree $S$ of $T$, the $S$-fixing group 
\begin{equation*}
	G_S := \{ g \in G \mid gx=x\; \forall x \in V(S) \}.
\end{equation*}
is a compact open subgroup of $G$, and the set
\begin{align*}
	\lbrace G_S\mid S \mbox{ finite, complete subtree of } T \rbrace
\end{align*}
forms a neighbourhood basis of the identity element in $G$.

We first recall the following result, which ensures that whenever $T$ is a locally finite tree and $G$ acts transitively on the boundary $\partial T$ of $T$, the tree is automatically semi-homogeneous. For a proof, we refer to \cite[Proposition 4]{amann}.
\begin{prp} \label{prp:semihomogeneous}
Let $T$ be a locally finite tree such that its boundary $\partial T$ consists of at least three elements. Let $G$ be a non-compact, closed subgroup of $\mathrm{Aut}(T)$ that acts transitively on $\partial T$. Then for every vertex $x$ of $T$, the stabiliser group $G_x$ acts transitively on $\partial T$.

Moreover, the tree $T$ is semi-homogeneous and $G$ has at most two orbits. In case $G$ has two orbits, the $G$-orbits in $V(T)$ are $\lbrace z \in V(T) \mid d_c(z,x)  \text{ is even} \rbrace$ and $\lbrace z \in V(T) \mid d_c(z,x) \text{ is odd} \rbrace$ for some vertex $x\in V(T)$.
\end{prp}
The following result is a consequence of Proposition \ref{prp:semihomogeneous}. It goes back to \cite{MR0578650} (see also \cite[Section II.4]{MR1152801}, \cite[Proposition 7]{amann}).
\begin{cor}\label{cor:tree_GP}
	Let $G$ be a non-compact, closed subgroup of $\Aut(T)$ which acts transitively on $\partial T$. Furthermore, let $x\in V(T)$ be a vertex. Then for every $s\in G$, the identity
	\begin{align*}
		\lbrace t\in G \mid d_c(x,tx) = d_c(x,sx) \rbrace = G_x s G_x
	\end{align*}
	holds.
	In particular, $(G,G_x)$ is a Gelfand pair.
\end{cor}
\begin{proof}
	Let $s\in G$, and let $t\in G$ be such that $d_c(x,tx) = d_c(x,sx)$. By Proposition \ref{prp:semihomogeneous}, we know that $G_x$ acts transitively on $\partial T$. Hence, it acts transitively on the sphere $\lbrace y\in V(T)\mid d_c(x,y)=d_c(x,sx) \rbrace$. Therefore, there is an
	element $k\in G_x$ with $tx=ksx$. This implies that $t\in G_x s G_x$ and therefore
	$\lbrace t\in G \mid d_c(x,tx)=d_c(x,sx)\rbrace \subset G_x s G_x$. The other
	inclusion is straightforward. Moreover, the identity that we just proved
	directly implies that $s^{-1}\in G_x s G_x$ for all $s\in G$. It is well known that this implies that $(G,G_x)$ is a Gelfand pair (see e.g.~\cite[Proposition 8.1.3]{MR2328043}).
\end{proof}

\subsection{Tits' independence property}
Let $T$ be a semi-homogeneous tree. Let  $e \in E(T)$ be an edge of $T$, and let $\pi_{e}\colon V(T)\to e$ be
the nearest point projection onto $e$.
The $e$-fixing group 
\[
	G_{e} = \lbrace s\in G\mid sx = x \; \forall x\in e\rbrace
\]
of a subgroup $G$ of $\mathrm{Aut}(T)$ preserves, for each $x\in e$, the set 
\[
	T_x=\pi_{e}^{-1}(x) =\lbrace  y\in V(T)\mid d_c(y,x)\leq d_c(y,z) \; \forall z\in e\rbrace.
\]
Let $F_x$ be the image of the restriction map 
\[
	\Phi_x\colon G_{e}\to \mathrm{Sym}(T_x),\, s\mapsto s\vert_{T_x}.
\] 
Then the map $ \Phi_{e}\colon G_{e}\to \prod_{x\in e}F_x,\, s\mapsto \Phi_x(s)$
is an injective group homomorphism.
\begin{dfn}
	A closed subgroup $G$ of $\mathrm{Aut}(T)$ has Tits' independence property if for each edge $e\in E(T)$, the homomorphism
	\begin{align*}
	\Phi_{e}\colon G_{e}\to \prod_{x\in e}F_x
	\end{align*}
	is an isomorphism.
\end{dfn}
This property goes back to Tits \cite{MR0299534}. The definition given here is a characterisation of the property for \emph{closed} subgroups of $\mathrm{Aut}(T)$ (see \cite[Section 1.2]{amann}). Tits introduced this property in order to study simple subgroups of $\mathrm{Aut}(T)$. However, the independence property also has far-reaching consequences for the asymptotic behaviour of matrix coefficients
of certain irreducible unitary group representations, as we will see later. 

The automorphism group $\mathrm{Aut}(T)$ as well as the trivial group have
the independence property. Less trivial examples are provided by the 
Burger-Mozes universal groups \cite{MR1839488}, which we briefly recall here.
\begin{exm}
	We assume $T$ to be a homogeneous tree of degree $d=d_0=d_1\geq 3$. Let $l\colon E(T)\to \lbrace 1,\cdots, d\rbrace$ be a legal
	labelling of $T$, i.e.~the map $l_x:=l\vert_{E(x)}\colon E(x)\to \lbrace 1,\ldots,d\rbrace$,
	where $E(x)$ denotes the set of edges containing $x\in V(T)$, is a bijection, and
	$l_x(e) = l_y(e)$ for every edge $e=\lbrace x,y\rbrace \in E(T)$.
	
	For every automorphism $s\in \mathrm{Aut}(T)$ and any vertex
	$x\in V(T)$ the composition $c(s,x)=\ell_{gx} \circ s \circ \ell_x^{-1}$ defines
	an element in the symmetric group $\mathrm{Sym}(\lbrace 1,\ldots, d\rbrace)$. 
	
	Let $F$ be a subgroup of $\mathrm{Sym}(\lbrace 1,\ldots,d\rbrace)$. Then
	\begin{align*}
		U^{(\ell)}(F) = \lbrace s\in \mathrm{Aut}(T)\mid
		c(s,x)\in F \; \forall x\in V(T) \rbrace
	\end{align*}
	forms a closed subgroup of $\mathrm{Aut}(T)$. The group
	$U^{(\ell)}(F)$ acts transitively on the vertices of $T$ and has Tits'
	independence property. Moreover, the group $U^{(\ell)}(F)$ acts transitively on the boundary of $T$ if and only if $F$ acts
	$2$-transitively on $\lbrace 1,\ldots,d\rbrace$. For details on these groups, we refer to \cite{MR1839488}.
\end{exm}
\begin{exm}
Let us mention the existence of closed
 subgroups of $\mathrm{Aut}(T)$ that do not satisfy Tits' independence property. One class of examples is given by the projective special linear groups $\mathrm{PSL}(2,\mathbb{Q}_p)$ over the $p$-adic numbers acting on its Bruhat-Tits tree (see \cite[Example 6.33]{garridoAutomorphismGroupsTrees2018}).
\end{exm}

\section{Spherical sub-$C^*$-algebras of $L^{p+}$-group $C^*$-algebras} \label{sec:sphericalsubcstaralgebras}

In this section, let $T$ be a semi-homogeneous tree of degree $(d_0,d_1)$ with $d_0,d_1\geq 2$ and $d_0+d_1 \geq 5$, and let $G$ be a non-compact, closed subgroup $G$ of $\Aut(T)$ that acts transitively on the boundary $\partial T$ of $T$. Fix a vertex $o\in V(T)$, let
\[
	K:=G_o=\lbrace s\in G\mid so=o\rbrace
\]
be the stabiliser group of $o$, and let $\mu_G$ be the Haar measure on $G$ satisfying $\mu_G(K) =1$.

Without loss of generality, we assume that $\mathrm{deg}_T(o) = d_0$. Furthermore, let
\begin{align*}
	\delta = (d_0-1)(d_1-1),
\end{align*}
and let $\kappa$ be the number of $G$-orbits in $V(T)$. By Proposition \ref{prp:semihomogeneous}, the action of $G$ has at most
two orbits in $V(T)$. Hence, we have $\kappa \in \lbrace 1,2\rbrace$ and
$Go=\lbrace x\in V(T)\mid d_c(o,x) \in \kappa \N_0\rbrace$.

This section is mainly aimed at describing the spherical sub-$C^*$-algebras 
$C^*_{L^{p+}}(K\backslash G/K)$ of $C^*_{L^{p+}}(G)$, where $p$ belongs to $[2,\infty]$. The main theorems of this article are also proved in this section.

\subsection{Asymptotic behaviour of spherical functions}
For the description of the spherical sub-$C^*$-algebras, we first need an understanding of the asymptotic behaviour of the spherical functions for the Gelfand pair $(G,K)$.

Let  $\vert \cdot \vert \colon G\to [0,\infty)$ be the function
given by $\vert s\vert= d_c(so,o)$ for $s\in G$.
Note that $\vert \cdot\vert$ is $K$-bi-invariant and only takes values
in $\kappa\N_0$.
\begin{prp}\label{GA:prop:spherical_algebra_auto_grp}
	For every $K$-bi-invariant function $f\in C_c(K\backslash G/K)$, there is
	a unique function $\dot{f}\in C_c(\kappa\N_0)$ such that $f = \dot{f}\circ \vert \cdot \vert$.
	The map
	\begin{align*}
		C_c(K\backslash G/K)\to C_c(\kappa \N_0), \; f\mapsto \dot{f}
	\end{align*}
	is an isomorphism of vector spaces.	Furthermore, for every $f\in C_c(K\backslash G/K)$, the following holds:
	\begin{align*}
	\int f\mathrm{d}\mu_G = \dot{f}(0)+ \frac{d_0}{d_0-1}\sum_{r\in \kappa\N} \delta^{\frac{r}{2}}\, \dot{f}(r)
	\end{align*}
\end{prp}
\begin{proof}
	For the first statement, suppose that
 $f\in C_c(K\backslash G/K)$ is a $K$-bi-invariant function. Then $f$
factors as a right-$K$-invariant function through the map $G\to Go,\, s\mapsto so$.  
Since, moreover, $f$ is left-$K$-invariant, and
since $K$ acts transitively on $\partial B_r(o) = \lbrace x\in V(T)\mid d_c(o,x) = r\rbrace$
for all $r\in \kappa\N_0$, the function $f$, identified as a function on $Go$, is constant on $\partial B_r(o)$ for all $r\in \kappa \N_0$. Therefore, there is a unique function $\dot{f}\in C_c(\kappa\N_0)$
such that $f = \dot{f}\circ \vert \cdot \vert$. This completes the proof of the first statement.
The second statement is evident. Hence, it remains to prove the last
statement. Again, suppose that $f$ is a $K$-bi-invariant function on $G$ with compact support.
Then
\begin{align*}
	\int	f\,\mathrm{d}\mu_G = \sum_{x\in Go}f'(x) = \sum_{r\in \kappa \N_0} \vert \partial B_r(o)\vert \dot{f}(r) =\dot{f}(0) + \frac{d_0}{d_0-1} \sum_{r\in \kappa \N} \delta^{\frac{r}{2}}\, \dot{f}(r),
\end{align*}
where $f'\in C_c(Go)$ denotes the unique function satisfying $f(s)=f'(so)$
for all $s\in G$. Here, we used that the cardinality $\vert\partial B_r(o)\vert$ of $\partial B_r(o)$
for $r\in \kappa\N$ is equal to $\frac{d_0}{d_0-1}\delta^{\frac{r}{2}}$.
Furthermore, we used that there is a $G$-invariant  measure $\mu_{G/K}$ on $G/K$
such that
\begin{align*}
	\int f\,\mathrm{d}\mu_G = \int_{G/K}\int_K f(xk)\mathrm{d}\mu_G(k)\,\mathrm{d}\mu_{G/K}(x)
\end{align*}
for all $f\in C_c(G)$. By identifying $G/K$ with $Go$, it is easy to see that
\begin{align*}
	\int_{G/K} f(so)\mathrm{d}\mu_{G/K}([s]) = \sum_{x\in Go}f(x)
\end{align*}
for every function $f\in C_c(Go)$. 
\end{proof}
The functions in the following proposition play a key role in the analysis of the asymptotic behaviour of the spherical functions for $(G,K)$.

\begin{prp}\label{GA:prop:length_tree_asymp}
	Let $z\in \C$. The continuous function
	\begin{align*}
		h_z \colon G\to \C,\, s\mapsto \delta^{-\frac{1}{2}z\vert s\vert }
	\end{align*}
	is $K$-bi-invariant and positive definite whenever $z\in [0,\infty)$. Moreover,
	for $p\in [1,\infty)$, the function $h_z$ belongs to $L^p(G)$ if and only if $\mathrm{Re}\, z \in 
	\left(\frac{1}{p},\infty\right)$.
\end{prp}
\begin{proof}
	We have already shown above that $\vert \cdot\vert$ is $K$-bi-invariant. Also, it is known that $\vert \cdot\vert$ is a conditionally negative definite function on $G$ (see e.g. \cite[Example C.2.2]{bekkaKazhdanProperty2008a}). It is therefore a direct consequence of Schoenberg's theorem (see e.g. \cite[Theorem C.3.2]{bekkaKazhdanProperty2008a}) that $h_z$ is positive definite if $z\in[0,\infty)$.
	
	Now assume that
	$z\in \C$ is an arbitrary complex number. Using Proposition \ref{GA:prop:spherical_algebra_auto_grp}, we obtain
	\begin{align*}
		\int \vert h_z(s)\vert^p \mathrm{d}\mu_G(s) -1 &= 
		\frac{d_0}{d_0-1}\sum_{r\in \kappa \N}\delta^{\frac{r}{2}} \,\vert \delta^{-\frac{1}{2}zr}\vert^p \\
		&=\frac{d_0}{d_0-1} \sum_{r\in \kappa \N}\delta^{\frac{r}{2}} \, \delta^{-\frac{1}{2}\mathrm{Re}(z)rp}\\
		&=
		\frac{d_0}{d_0-1}\sum_{r\in \kappa \N}\delta^{\frac{1}{2}(1-\mathrm{Re}(z)p)r} 
	\end{align*}
	for $p\in [1,\infty)$.
	The right-hand side of the equation converges if and only if
	$1-\mathrm{Re}(z) p  <0$.
	This implies that the above integral is finite if and only if
	$\mathrm{Re}\, z\in \left(\frac{1}{p},\infty\right)$.
\end{proof}
The number $\delta$ in the definition of $h_z$ may seem to have been chosen randomly. However, the reason for this choice is, as we will see below, that spherical functions for $(G,K)$ are linear combinations of elements in $\lbrace h_z\mid z\in \C \rbrace$.

We recall two results regarding the Gelfand pair $(G,K)$. The first can be found in \cite[p.~31]{amann}.
	\begin{prp}\label{EX:prop:generator_self_adjont}
		Let $s\in G$ be an element with $\vert s \vert = m \in \N_0$. The function
		\begin{align*}
			\mu_m = \frac{1}{\vert \partial B_{m} (o) \vert} \mathbf{1}_{KsK},
		\end{align*}
		 where
		$\partial B_{m}(o) =
		\lbrace x\in V(T) \mid d_c(x,o) = m\rbrace $,
		is a self-adjoint element of $C_c(K\backslash G /K)$.
	\end{prp}
	The second result we recall can be found in \cite[p.~32]{amann}.
	\begin{prp}
		 The set $\lbrace \mu_0,\, \mu_\kappa \rbrace$, where $\kappa$ denotes, as before, the number of $G$-orbits, generates
		the ${}^*$-algebra $C_c(K\backslash G/K)$.
		Moreover, the following identities hold:
		\begin{equation*}
			\mu_1\ast \mu_n=
			\begin{dcases}
				\mu_1 & \text{if } n=0\\
				\frac{1}{d_0}\mu_{n-1} + \frac{d_0-1}{d_0}\mu_{n+1} & \text{if } n\in \N
			\end{dcases}
		\end{equation*}
		if $\kappa = 1$, and
		\begin{equation*}
			\mu_2 \ast \mu_n =
			\begin{dcases}
				\mu_2 & \text{if } n=0\\
				\frac{1}{d_0(d_1-1)}\mu_{n-2} + \frac{d_1-2}{d_0(d_1-1)}\mu_{n}+\frac{d_0-1}{d_0} \mu_{n+2}& \text{if } n\in 2\N
			\end{dcases}
		\end{equation*}
		 if $\kappa = 2$.
	\end{prp}
Thus, the investigation of $C^*$-completions of the algebra $C_c(K \backslash G / K)$ heavily relies on the investigation of the spectrum of $\mu_\kappa$. Although not essential for the following discussion, the following result might be helpful for a first understanding.

Recall that for a spherical Radon measure
	$\chi$ for $(G,K)$, there is a spherical function
	$\varphi$ for $(G,K)$ such that $\chi(f) = \int f(s)\varphi(s^{-1})\mathrm{d}\mu_G(s)$ for
	all $f\in C_c(K\backslash G/K)$.
	\begin{prp}
		The spectrum $ \sigma_{C^*(K\backslash G/K)}(\mu_\kappa)$ of $\mu_\kappa$ in $C^*(K\backslash G/K)$ is contained in the compact interval
		$[-1,1]$. Moreover, $1\in \sigma_{C^*(K\backslash G/K)}(\mu_\kappa)$.
	\end{prp}
	\begin{proof}
		Let $\lambda \in \sigma_{C^*(K\backslash G/K)}(\mu_\kappa)$. There is
		a character $\chi:C^*(K\backslash G/K)\to \C$ with $\chi(\mu_\kappa) =
		\lambda$. Pulling back $\chi$
		to $C_c(G)$ leads to a spherical Radon measure $\tilde{\chi}:C_c(G)\to \C$.
		Hence, there is a positive definite, spherical function $\varphi$ for
		$(G,K)$ with $\tilde{\chi}(f)=\int f(s) \varphi(s^{-1})\mathrm{d}\mu_G(s)$
		for $f\in C_c(K\backslash G/K)$. Since $\varphi$ is bounded by $1$, we
		have
		\begin{align*}
			\vert \lambda \vert \leq \int \vert \mu_\kappa (s)\varphi(s^{-1})\vert\,
			\mathrm{d}\mu_G(s) \leq \int \mu_\kappa(s)\,\mathrm{d}\mu_G(s)
			=\tau_0(\mu_\kappa) =1,
		\end{align*}
		where $\tau_0$ is the trivial group representation of $G$.
	\end{proof}
	\begin{prp}\label{EX:tree:prop_eigen_eq}
		A function $\varphi \in C(K\backslash G/K)$ with $\varphi(e) =1$ is spherical 
		for $(G,K)$ if and only if
		there exists a complex number $\gamma_\varphi\in \mathbb{C}$ such 
		that
		\begin{align*}
			\mu_\kappa \ast \varphi =  \gamma_\varphi \,\varphi.
		\end{align*}
		In addition, a spherical function $\varphi$ is uniquely determined by its
		eigenvalue $\gamma_\varphi$, and the eigenvalue $\gamma_\varphi$ is a real number if
		$\varphi$ is positive definite.
	\end{prp}
	\begin{proof}
		If $\varphi$ is spherical, the existence of a complex number $\gamma_\varphi$
		with $\mu_\kappa \ast \varphi = \gamma_\varphi \,\varphi$ follows from \cite[Theorem 8.2.6]{MR2328043}. The other direction follows from the
		fact that the set $\lbrace \mu_0, \mu_\kappa\rbrace$ generates $C_c(K\backslash G/K)$.
		Indeed, suppose that there is a complex number $\gamma_\varphi$ with
		$\mu_\kappa \ast \varphi = \gamma_\varphi \,\varphi$. Since
		$\lbrace \mu_0,\mu_\kappa\rbrace$ generates $C_c(K\backslash G/K)$, for every
		function $f\in C_c(K\backslash G/K)$, there are complex numbers $a_0,\dots, a_n \in \C$
		such that $f= \sum_{i=0}^n a_i (\mu_\kappa)^i$. Thus, we have
		\begin{align*}
			f\ast \varphi = \sum_{i=0}^n a_i (\mu_\kappa)^i \ast \varphi = \left( \sum_{i=0}^n a_i \gamma_\varphi^i \right) \varphi.
\end{align*}		
		This implies that $\varphi$ is a spherical function for $(G,K)$ (see e.g.
		\cite[Theorem 8.2.6]{MR2328043}).
		
		 The second statement is again a consequence of the fact that $\lbrace \mu_0,\mu_\kappa\rbrace$ generates
		$C_c(K\backslash G/K)$. Indeed, let $\varphi_1$ and $\varphi_2$
		be two spherical functions for $(G,K)$ with
		$\mu_\kappa \ast \varphi_i = \lambda \varphi_i$ for $i\in\lbrace 1, 2\rbrace$, where
		$\lambda\in \C$,  and let $\chi_i$ be the spherical 
		Radon measure
		corresponding to $ \varphi_i$. Then we have $\chi_1(\mu_\kappa) = 
		\mu_\kappa\ast \varphi_1(e) = \mu_\kappa \ast \varphi_2(e) 
		=\chi_2(\mu_\kappa)$.
		Linearity and multiplicativity now lead to the identity $\chi_1\vert_{C_c(K\backslash G/K)} = \chi_2\vert_{C_c(K\backslash G/K)}$.
		This implies that $\varphi_1 = \varphi_2$.
		Finally, assume that $\varphi$ is a positive definite, spherical function for $(G,K)$. Then
		we have $\overline{\varphi(s)}= \overline{\varphi(s^{-1})} =  \varphi(s)$ for all
		$s\in G$. Hence, $\varphi$ is a real-valued function. Therefore, 
		$\gamma_\varphi = \mu_\kappa \ast \varphi(e)$ is a real number.
	\end{proof}
	\begin{dfn}
		In the following, let $\varphi_\gamma$ denote the spherical
		function for $(G,K)$ satisfying
		\begin{align*}
			\mu_\kappa \ast \varphi_\gamma = \gamma \, \varphi_\gamma,
		\end{align*}
		where $\gamma \in \C$.
	\end{dfn}
	\begin{dfn}
		For a complex number $z \in \mathbb{C}$, we define
		\begin{equation*}
			\gamma_o(z) := 
			\begin{dcases}
				\frac{1}{d_0}\left((d_0-1)^z + (d_0-1)^{1-z}\right) & \mbox{if }  \kappa=1, \\
				\frac{1}{d_0(d_1-1)}\left( \delta^{z} + \delta^{1-z} + (d_1-2)\right) & \mbox{if }	\kappa=2.
			\end{dcases}
		\end{equation*}
	\end{dfn}
	In the following, we follow the arguments of \cite{MR0710827} to study the asymptotic behaviour of spherical functions for $(G,K)$.
	\begin{prp}\label{GA:prop:fund_sol_tree}
		Let $z\in \mathbb{C}$ be a complex number. The $K$-bi-invariant function
		$h_z \colon G\to \C$
		satisfies 
		\begin{align*}
			\mu_\kappa \ast h_z (s) = \gamma_o(z) h_z(s)
		\end{align*}
		for all $s$ with $\vert s\vert \neq 0$.
	\end{prp}
	\begin{proof}
		There are two cases to consider: the case that $G$ acts transitively on $V(T)$ and
		the case that $G$ does not act transitively on $V(T)$. We present the computation for the latter case, so let us assume that $G$ does not act transitively on $V(T)$.
		
		Let $\phi_z\colon 2\N_0\to \R$ be the function given by
		$\phi_z(m) = \delta^{-\frac{1}{2}z m} \vert \partial B_m(o) \vert
		$ for $m\in 2\N_0$. 
		Then the sequence $(h_z^{(n)})_{n\in \N_0} = \left(\sum_{k= 0}^n \phi_z(2k)\mu_{2k}\right)_{n\in \N_0}$ converges to $h_z$ uniformly on compact subsets of $G$, which, in turn, implies the uniform convergence of $(\mu_2\ast h_z^{(n)})_{n\in \N_0}$ to $\mu_2\ast h_z$ on compact subsets of $G$.	We have
		\begin{align*}
			\mu_2 \ast h_z =&
			\sum_{m\in 2\N_0}  \phi_z(m) \mu_2\ast \mu_m\\
			=& \sum_{m\in 2\N}
			\phi_z(m) 
			\left(\frac{1}{d_0(d_1-1)}\mu_{m-2} + \frac{d_1-2}{d_0(d_1-1)}\mu_{m}+\frac{d_0-1}{d_0} \mu_{m+2} \right) \\ &+ \phi_z(0)\mu_2
			\\
			= &\sum_{m\in 2\N_0}\frac{1}{d_0(d_1-1)}\phi_z(m+2)\mu_m
			+ \sum_{m\in 2\N}\frac{d_1-2}{d_0(d_1-1)}\phi_z(m) \mu_m
			\\&+  \sum_{m \in \lbrace 4,6,\ldots\rbrace}\frac{d_0-1}{d_0}\phi_z(m-2) \mu_m
			+\phi_z(0)\mu_2.
		\end{align*}
		Since the evaluation at any point in $G$ is continuous with respect to the
		uniform convergence on compact subsets of $G$, the assertion follows easily.
		Indeed, recall that for $m\in 2\N$ we have $\vert\partial B_m(o)
		\vert = \frac{d_0}{d_0-1}\,\delta^{\frac{m}{2}}$.  Hence, we have
		\begin{align*}
			\phi_z(m) = \frac{d_0}{d_0-1}\delta^{\frac{1}{2}(1-z)m}
		\end{align*}
		for $m\in 2\N$. Thus, for $s\in G$, we obtain
		\begin{align*}
			&\vert \partial B_2(o)\vert \mu_2\ast h_z(s)\\ 
			&= 
				\frac{1}{d_0(d_1-1)}\frac{d_0}{d_0-1}\delta^{\frac{1}{2}(1-z)(2+2)}
				+ \frac{d_1-2}{d_0(d_1-1)}\frac{d_0}{d_0-1}\delta^{\frac{1}{2}(1-z)2}+1\\
			&=
			\frac{1}{d_0(d_1-1)}\left( \delta^{1-z}+d_1-2 
			+ (d_1-1)(d_0-1)\delta^{z-1}\right)
			\frac{d_0}{d_0-1}\delta^{\frac{1}{2}(1-z)2}\\
			&= 
			\vert \partial B_2(o)\vert \gamma_o(z) \delta^{-\frac{1}{2}z2} = 
			\vert \partial B_2(o)\vert\gamma_o(z) h_z(s)
		\end{align*}
		if $m=\vert s\vert =2$, and
	\begin{align*}
		&\vert \partial B_m(o)\vert \mu_2\ast h_z(s)\\
		&=\frac{1}{d_0(d_1-1)}\frac{d_0}{d_0-1}\delta^{\frac{1}{2}(1-z)(m+2)}+
			\frac{d_1-2}{d_0(d_1-1)}\frac{d_0}{d_0-1}\delta^{\frac{1}{2}(1-z)m}+
			\delta^{\frac{1}{2}(1-z)(m-2)}\\
		&= \frac{1}{d_0(d_1-1)}\left( \delta^{1-z}+ d_1-2+ (d_1-1)(d_0-1)\delta^{z-1}\right)\frac{d_0}{d_0-1}\delta^{\frac{1}{2}(1-z)m}\\
		&= \vert \partial B_m(o)\vert \gamma_o(z) h_z(s)
	\end{align*}
	if $m=\vert s\vert \geq 4$.
	\end{proof}
	
	\begin{lem}[{cf. \cite[Theorem 2.2]{MR0710827}}]\label{EX:lemma_representation_of_sph_func}
		Let $z\in \mathbb{C}$ be a complex number with $\delta^{-\frac{(1-z)}{2}\kappa}\neq 
		\delta^{-\frac{z}{2}\kappa}$. Then there are constants $c(z)\neq 0$ depending on $z$ such that
		\begin{align*}
			\varphi_{\gamma_o(z)} = c(z) h_z + c(1-z) h_{1-z}.
		\end{align*}
	\end{lem}
	\begin{proof}
		We again assume that $G$ does not act transitively on $T$. The other case is analogous.
		 Since $\delta^{-(1-z)}\neq 
		\delta^{-z}$,
		the matrix
		\begin{equation*}
			\begin{pmatrix}
			1 & 1 \\ 
			\delta^{-z} & \delta^{-(1-z)}
			\end{pmatrix} 
		\end{equation*}
		is invertible. Since $\gamma_o(z) = \gamma_o(1-z)$, the solution $(x(z),y(z))$ of the equation
		\begin{equation}\label{EX:eq:spherical_func_as_length_func}
			\begin{pmatrix}
			1 & 1 \\ 
			\delta^{-z} & \delta^{-(1-z)}
			\end{pmatrix} 
			\begin{pmatrix}
				x(z) \\ y(z)
			\end{pmatrix}
			= \begin{pmatrix}
				1 \\ \gamma_o(z)
			\end{pmatrix},
		\end{equation}
		must be of the form $x(z) = y(1-z)$, as an exchange of $z$ and $1-z$ in
		\eqref{EX:eq:spherical_func_as_length_func} shows.
		So let $c(z) = x(z)$. Then $c(z) h_z(s) + c(1-z)h_{1-z}(s) = 1$ and
		$c(z) h_z(t) + c(1-z) h_{1-z}(t) = \gamma_o(t)$ for every $s,t\in G$
		with $\vert s \vert = 0$ and $\vert t\vert = 2$. Let
		$\varphi'_z = c(z) h_z + c(1-z) h_{1-z}$. In order to prove that
		$\varphi'_z= \varphi_{\gamma_o(z)}$, it is sufficient, by Proposition
		\ref{EX:tree:prop_eigen_eq}, to show that $\mu_2 \ast \varphi'_z
		= \gamma_o(z) \varphi'_z$. Let $s\in G$ be an element with $\vert s\vert
		\neq 0$. Then 
		\begin{align*}
			\mu_2\ast \varphi'_z(s) &= c(z) \mu_2 \ast h_z(s) + c(1-z) \mu_2\ast h_{1-z}(s)\\
			 &=	c(z) \gamma_o(z) h_z(s)+ c(1-z) \gamma_o(z) h_{1-z}(s)\\
			 &= \gamma_o(z) \varphi'_z(s),
		\end{align*}
		which completes the proof.
	\end{proof}
	\begin{lem}\label{EX:lemma:spher_func_asympt_tree}\label{GA:lem:asym_of_sphe_func_tree}
		Let $z\in \C$ be a complex number with  $\delta^{-\frac{(1-z)}{2}\kappa}\neq 
		\delta^{-\frac{z}{2}\kappa}$, let
		 $p\in (2,\infty)$, and let $q\in (1,2)$ be such that $\frac{1}{p}+\frac{1}{q}=1$. Then $\varphi_{\gamma_o(z)}$ belongs to $L^p(G)$
		if and only if $\mathrm{Re}\, z \in \left(\frac{1}{p},\frac{1}{q}\right)$.
	\end{lem}
	\begin{proof}
		Using Lemma \ref{EX:lemma_representation_of_sph_func}, it is easy to see
		that $\varphi_{\gamma_o(z)}$ is unbounded whenever $\mathrm{Re}\, z \not\in [0,1]$.
		Now, suppose that
		$\mathrm{Re}\, z \in \left(\frac{1}{p},\frac{1}{q}\right)$. 
		Note that $1-\mathrm{Re}\, z >1-\frac{1}{q} =\frac{1}{p}$. Hence, by Proposition \ref{GA:prop:length_tree_asymp}, the functions
		$h_z$ and $h_{1-z}$ belong to $L^p(G)$. Since, due to Lemma \ref{EX:lemma_representation_of_sph_func}, $\varphi_{\gamma_o(z)}$
		is a linear combination of $h_z$ and $h_{1-z}$, $\varphi_{\gamma_o(z)}$
		belongs to $L^p(G)$.
		
		Suppose that $\mathrm{Re}\, z\in [0,\frac{1}{p}]$. Then
		we have $1-\mathrm{Re}\, z \geq 1-\frac{1}{p} =\frac{1}{q} > \frac{1}{p}$.
		By Proposition \ref{GA:prop:length_tree_asymp}, $h_z$
		does not belong to $L^p(G)$ while $h_{1-z}$ belongs to $L^p(G)$.
		This implies that $\varphi_z\not \in L^p(G)$. The remaining case follows from a similar argument.
	\end{proof}	
	In the analysis of the asymptotic behaviour of spherical functions for
	$(G,K)$, we have already found some necessary conditions for positive
	definiteness of these functions. We now recall Amann's classification of positive definite, spherical functions for automorphism groups with independence property.
	
	First, we recall another description of the spherical functions for $(G,K)$.
	For this purpose, let $\omega\in \partial T$ be
	a boundary point, and let $G_\omega = P$ be the stabilizer group of $\omega$.
	Since $G$ acts transitively on $\partial T$, and since $\partial T$ is compact,
	the continuous map $G\to \partial T,\, s\mapsto s\omega$ induces
	a $G$-equivariant homeomorphism $G/P\to \partial T$.
	Therefore, there is a unique $K$-invariant and
	quasi-$G$-invariant Radon probability measure $\nu_o$
	on $\partial T$, and we have
	\begin{align*}
		\int_K f(k\omega)\mathrm{d}\mu_G(k) = \int_{\partial T} f(\omega')\mathrm{d}\nu_o(\omega')
	\end{align*}
	for all $f\in C(\partial T)$.
	
	Note that for every $s\in G$, the probability measure
	$s(\nu_o)= \nu_{so}$ is the $sKs^{-1}=  G_{so}$-invariant Radon probability measure. The function
	\begin{align*}
		P_o \colon G\times \partial T \to (0,\infty),\, (s,\omega')\mapsto
		\frac{\mathrm{d}\nu_{so}}{\mathrm{d}\nu_o}(\omega')
	\end{align*}
	is called the Poisson kernel, which is explicitly given by
	\begin{align*}
		P_o(s,\omega') = \delta^{\frac{1}{2}\langle so,o; \omega'\rangle},
	\end{align*}
	where $\langle so,o;\omega'\rangle = d_c(o,u)-d(so,u)$ for any
	$u\in [o,\omega')\cap[so,\omega')$ (see \cite[Proposition 1.8.4]{choucrounAnalyseHarmoniqueGroupes1994}).
	\begin{dfn}
		For $z\in \C$, we define $\pi_z\colon G\to \mathcal{B}(L^2(\partial T,\nu_o))$ to be the group representation
		of $G$ 
		given by  
		\begin{align*}
			\pi_z(s)f(\omega) = P_o^z(s,\omega)f(s^{-1}\omega)
		\end{align*}
		for $s\in G$, $f\in L^2(\partial T,\nu_o)$ and $\omega \in \partial T$.
	\end{dfn}
	A proof of the following result can be found in \cite[Lemma 42]{amann}.
	\begin{thm} \label{EX:thm_spherical_func_on_bdy_tree}
		Let $G$ be a non-compact, closed subgroup of $\mathrm{Aut}(T)$ that acts transitively on $\partial T$.
		Then for each $z\in \C$, the function
		\begin{align*}
			 G \to \C,\, s\mapsto \int P_o^z(s,\omega)\,\mathrm{d}\nu_o(\omega)
			= \langle \pi_z(s) \mathbf{1}_{\partial T},\mathbf{1}_{\partial T}\rangle
		\end{align*}
		is a spherical function for $(G,K)$, and every spherical function for $(G,K)$ is of this form. More precisely,
		\begin{align*}
			\varphi_{\gamma_o(z)}(s) = \int P_o^z(s,\omega)\mathrm{d}\nu_o(\omega)
		\end{align*}
		 for all $s\in G$.
	\end{thm}
	The following theorem (see \cite[Theorem 2]{amann}) characterises the positive definite spherical functions on automorphism groups satisfying Tits' independence property.
	\begin{thm} \label{GA:thm:tree_classification_pd_sph_func}
	Let $G$ be a non-compact, closed subgroup of $\mathrm{Aut}(T)$ that acts transitively on $\partial T$ and that satisfies Tits' independence property. Let
		\begin{align*}
			\mathcal{P}= \left( \left\{\frac{1}{2}\right\}+i\left[0, \frac{2\pi}{\kappa \log \delta} \right]\right)
			\bigcup \left(\left[ 0, \frac{1}{2}\right)+ i\left\{ 0 ,\frac{2\pi}{\kappa\log \delta} \right\} \right).
		\end{align*}
		Then the map
		\begin{align*}
			\mathcal{P}\to \mathcal{SP}(K\backslash G/K), \; z \mapsto \varphi_{\gamma_o(z)}
		\end{align*}
		is a bijection into the set of positive definite, spherical functions 
		$\mathcal{SP}(K\backslash G/K)$ for $(G,K)$.
	\end{thm}
The classification of positive definite, spherical functions for $(G,K)$ can be reformulated as follows.
	\begin{rmk}[{cf. \cite[Theorem 2]{amann}}]\label{GA:rem:tree_char_vs_spec_of_mu}
		Let $G$ be a non-compact, closed subgroup of $\mathrm{Aut}(T)$ that acts transitively on $\partial T$ and that satisfies Tits' independence property.
		 The range of the map
		\begin{equation*}
		\mathcal{P}
			\to [-1,1], \; z\mapsto \mu_\kappa \ast \varphi_{\gamma_o(z)}(e)
			= \gamma_o(z)
		\end{equation*}
		is the spectrum $\sigma_{C^*(K\backslash G/K)}(\mu_\kappa)$
		of $\mu_\kappa$ in $C^*(K\backslash G/K)$.
		The following holds:
		\begin{equation*}
			\sigma_{C^*(K\backslash G/K)}(\mu_\kappa)=
			\begin{dcases}
				\; [-1,1] & \mbox{ if } \kappa = 1,\\
				\; \left[ -\frac{2+(d_0-2)(d_1-1)}{d_0(d_1-1)},1\right]
				& \mbox{ if } \kappa = 2.
			\end{dcases}
		\end{equation*}
Moreover, in the case $\kappa=1$, we have that $\varphi_{-1}=\varphi_{\gamma_o(i\frac{2\pi}{\log \delta})}$ is equal to the group 
		homomorphism $G\to \mathbb{C},\; s\mapsto (-1)^{\vert s\vert}$. Let $\pi'_\lambda\colon G\to \mathcal{U}(\mathcal{H}_\lambda)$
		be the spherical representation of $\varphi_\lambda$ for $\lambda \in [-1,1]$, i.e., $\pi'_\lambda\colon G\to \mathcal{U}(\mathcal{H}_\lambda)$ is the irreducible unitary group representation of $G$ with
		 $K$-invariant unit vector $\xi\in \mathcal{H}_\lambda$ such that
		$\varphi_\lambda = (\pi'_\lambda)_{\xi,\xi}$.
		Then the unitary group representation 
		$\pi'_{-1}\otimes \pi'_{\lambda}$ is
		unitary equivalent to $\pi'_{-\lambda}$ (see
		\cite[p. 43]{amann}).
	\end{rmk}
The previous results depend on the choice of the vertex $o\in V(T)$, i.e.~on the chosen Gelfand pair. This dependence, however, does not apply to the spherical unitary dual, as the following theorem shows.
	\begin{thm}[{\cite[Theorem 2]{amann}}]\label{GA:thm:independence_of_vertex}
	Let $G$ be a non-compact, closed subgroup of $\mathrm{Aut}(T)$ that acts transitively on $\partial T$ and that satisfies Tits' independence property.
	Let $o, o' \in V(T)$ be two vertices. Then $(G,G_o)$ and
	$(G, G_{o'})$ are Gelfand pairs, and the spherical unitary duals
	$(\widehat{G}_{G_o})_1$ and $(\widehat{G}_{G_{o'}})_1$ coincide.
	To be more precise, every spherical unitary representation for $(G,G_o)$ is
	a spherical unitary representation for $(G,G_{o'})$.
\end{thm}
	The following remark complements Lemma \ref{EX:lemma:spher_func_asympt_tree}.
	
	\begin{rmk}
		Let $G$ be a non-compact, closed subgroup of $\mathrm{Aut}(T)$ that acts
		transitively on $\partial T$ and that satisfies Tits' independence property.
		Then the positive definite spherical function
		\begin{align*}
			\varphi_{\gamma_o\left(\frac{1}{2}\right)}\colon G\to \C,\, s\mapsto \int P_o^{\frac{1}{2}}(s,\omega) \,\mathrm{d}\nu_o(\omega)
		\end{align*}
		belongs to $L^{2+}(G)$.
	\end{rmk}

	We close this paragraph with the description of the spherical sub-$C^*$-algebras
	of the $L^{p+}$-group $C^*$-algebras.
	\begin{prp}\label{GA:prp:spectrum_of_mu_in_radial_algs}
		Let $G$ be a non-compact, closed subgroup of $\mathrm{Aut}(T)$ that acts
		transitively on $\partial T$ and that satisfies Tits' independence property.
		Let $p\in [2,\infty)$, and let $q\in (1,2]$ be such that
		$\frac{1}{p}+\frac{1}{q}=1$. Then the spectrum
		$\sigma_{C^*_{L^{p+}}(K\backslash G/K)}(\mu_\kappa)$ of $\mu_{\kappa}$ in $C^*_{L^{p+}}(K\backslash G/K)$ is equal to
		\[
			\left\{ \gamma_o(z)\mid  z\in \mathcal{P}, \; \mathrm{Re}\;z
		\in \left[\frac{1}{p},\frac{1}{q}\right] \right\}.
		\]
	\end{prp}
	\begin{proof}
		Let $z\in \mathcal{P}$. From Theorem \ref{GA:thm:tree_classification_pd_sph_func}, it follows that the function $\varphi_{\gamma_o(z)}$ is a positive definite function. 
		Let us denote the GNS-construction of $\varphi_{\gamma_o(z)}$
		by $\pi'_{\gamma_o(z)}\colon G\to \mathcal{U}(\mathcal{H}'_{\gamma_o(z)})$. Lemma \ref{EX:lemma:spher_func_asympt_tree}
		implies that $\pi'_{\gamma_o(z)}$ extends to a
		*-representation of $C^*_{L^{p+}}(G)$ whenever
		$\mathrm{Re}\,z\in \left[\frac{1}{p},\frac{1}{q}\right]$. On the other hand, if $\mathrm{Re}\, z \not\in \left[\frac{1}{p},\frac{1}{q}\right]$, then $\pi'_{\gamma_o(z)}$ has a vector state, namely $\varphi_{\gamma_o(z)}$,
		that is not in $L^{p+}(G)$. Hence, $\pi'_{\gamma_o(z)}$  does not extend to
		$C^*_{L^{p+}}(G)$ by Theorem \ref{KS:cor_dual_Lpplus_algebra}. This implies that
		\begin{align*}
			\chi_{\varphi_{\gamma_o(z)}}\colon C_c(K\backslash G/K)\to \C,\,
			f\mapsto f\ast \varphi_{\gamma_o(z)}(e)
		\end{align*} 
		extends to a character on $C^*_{L^{p+}}(K\backslash G/K)$ if
		and only if $\mathrm{Re}\, z\in \left[\frac{1}{p},\frac{1}{q}\right]$. Hence the result follows from the identity
		\[
			\sigma_{C^*_{L^{p+}}(K\backslash G/K)}(\mu_\kappa)
			=
			\left\{ \chi_{\varphi_{\gamma_o(z)}}(\mu_\kappa)\mid
			z\in \mathcal{P}, \; \mathrm{Re}\,z \in \left[\frac{1}{p},\frac{1}{q}\right]\right\}.\]
	\end{proof}
	Note that Proposition \ref{GA:prp:spectrum_of_mu_in_radial_algs} immediately implies Theorem \ref{thm:groupcstaralgebrastrees}.

\subsection{Spherical, special and super-cuspidal representations}
We now elaborate more on the representation theory of the groups under consideration.

Let $T$ be a semi-homogeneous tree of degree $(d_0,d_1)$ with $d_0,d_1\geq 2$ and $d_0+d_1\geq 5$, and $G$ be a non-compact, closed subgroup of $\mathrm{Aut}(T)$ that acts transitively on $\partial T$. Let $\mathcal{C}$ be the set of finite, complete subtrees of $T$ endowed with the inclusion as ordering.
For every element $S\in \mathcal{C}$, the $S$-fixing group $G_S=\lbrace s\in G\mid sx=x \; \forall x\in V(S) \rbrace$ is a compact open subgroup of $G$. The set
\begin{align*}
	\lbrace G_S \subset G \mid S\in\mathcal{C}\rbrace
\end{align*}
forms a neighbourhood basis of the identity element of $G$. Note that $G_{S_2}\subset G_{S_1}$ whenever $S_1 \subset S_2$. For every element $S \in \mathcal{C}$, set
\begin{align*}
	p_{S} = \frac{1}{\mu_G(G_s)}\mathbf{1}_{G_S}\in C_c(G).
\end{align*} 
\begin{prp}
	The net $(p_S)_{S\in \mathcal{C}}\in C_c(G)^\mathcal{C}\subset C^*(G)^\mathcal{C}$ is a monotonically increasing approximate identity of $C^*(G)$ that consists of orthogonal projections.
\end{prp}
Let $\pi\colon G\to \mathcal{U}(\mathcal{H})$ be a unitary group representation of $G$. It follows from the continuity of $\pi_*\colon C^*(G) \to \mathcal{B}(\mathcal{H})$ that there is an element $S\in \mathcal{C}$ such that
the orthogonal projection 
$
	\pi_*\left(p_S\right)$
is non-trivial.
In particular, there is a minimal complete subtree $S\in \mathcal{C}$ of $T$ such that $\pi_*\left(p_S\right)\neq 0$. Let $M_v$ be the set
of minimal complete subtrees of $T$ such that $\pi_*\left(p_S\right)\neq 0$.
\begin{dfn}
	Suppose that $G$ satisfies Tits' independence property, and let
	$\pi\colon G\to \mathcal{U}(\mathcal{H})$ be an irreducible unitary
	group representation of $G$.
	\begin{enumerate}[(i)]
		\item $\pi\colon G\to \mathcal{U}(\mathcal{H})$ is called \emph{spherical} if there exists an element
		in $M_v$ that is exactly one vertex.
		\item $\pi\colon G\to \mathcal{U}(\mathcal{H})$ is called \emph{special} if it is not spherical and there exists an element in $M_v$ that is an edge.
		\item $\pi\colon G\to \mathcal{U}(\mathcal{H})$ is called \emph{super-cuspidal} if it	is neither spherical nor special.
	\end{enumerate}
\end{dfn}
	The following theorem is part of \cite[Theorem 2]{amann}.
	\begin{thm}\label{GA:thm:discrete_reps_asymptotic}
		Let $G$ be a non-compact, closed subgroup of $\mathrm{Aut}(T)$ that acts
		transitively on $\partial T$ and that satisfies Tits' independence property. Then the following holds:
		\begin{enumerate}[(i)]
		\item Every special representation of $G$ is an $L^2$-representation.
		\item Every super-cuspidal representation of $G$ is an $L^1$-representation.
		\end{enumerate}
	\end{thm}	
	\subsection{Decomposition of the unitary dual} As in the previous paragraph, we assume that	$T$ is a semi-homogeneous tree of degree $(d_0,d_1)$ with 
	$d_0,d_1\geq 2$ and $d_0+d_1 \geq 5$. We now prove a decomposition theorem of the unitary dual of automorphism groups of $T$.
	\begin{thm}\label{GA:thm:decomposition_grp_alg_auto_grp}
		Let $G$ be a non-compact, closed subgroup of $\mathrm{Aut}(T)$ that acts
		transitively on $\partial T$ and that satisfies Tits' independence property. Let $o\in V(T)$ be a fixed vertex, and let $K = G_o$. Then there is
		a closed ideal $C^*(G)_{\mathrm{disc}}$ in $C^*(G)$ (where ``disc'' stands for ``discrete'') such that
		the spectrum $\reallywidehat{C^*(G)_{\mathrm{disc}}}$ consists
		of equivalence classes of super-cuspidal and special representations.
		Furthermore, we have the following decomposition
		\begin{align*}
			C^*(G) = C^*(G)_{\mathrm{disc}}\oplus C^*(G,K),
		\end{align*}
		where $C^*(G,K)$ is the spherical ideal for $(G,K)$ in $C^*(G)$.
	\end{thm}	
	\begin{proof}
	From Remark \ref{GA:rem:GP_spherical_ideal_morita_eq}, it follows that $(\widehat{G}_K)_1$ is an open subset of $\widehat{G}$. In order to prove the theorem, it suffices to show that $(\widehat{G}_K)_1$ is also a closed subset
		of $\widehat{G}$.
		
		By Theorem \ref{GA:thm:independence_of_vertex}, every spherical representation of $G$ defines an equivalence class of $(\widehat{G}_K)_1$. Hence,
		the set $\widehat{G}\setminus (\widehat{G}_K)_1$ consists of
		equivalence classes of special and super-cuspidal representations.
		
		From Theorem \ref{GA:thm:discrete_reps_asymptotic}, we know that
		every super-cuspidal representation is an $L^1$-representation. 
		By \cite[Corollary 1]{dufloRegularRepresentationNonunimodular1976},
		every irreducible $L^1$-representation forms an open point in
		$\widehat{G}$. This implies that for every
		super-cuspidal representation $\pi\colon G\to \mathcal{U}(\mathcal{H})$ of $G$,
		the equivalence class $[\pi]\in \widehat{G}$ does not belong
		to the closure of $(\widehat{G}_K)_1$.
		
		Now, suppose that
		$\pi\colon G\to \mathcal{U}(\mathcal{H})$ is a special representation of $G$. Recall that
		$C^*(G)$ admits an approximate identity $(p_S)_{S\in \mathcal{C}}$
		of orthogonal projections given by $p_S= \frac{1}{\mu_G(G_S)}\mathbf{1}_{G_S}$,
		where $\mathcal{C}$ denotes the directed set consisting of finite, complete 		subtrees of $T$ ordered by inclusion. By assumption, there exists an edge $e\in E(T)$
		such that $\pi_*(p_e)\neq 0$. Without loss of generality, $o\in e$. Again
		by assumption, $\pi_*(p_o) = 0$. Note that $p_o\leq p_e$. Therefore, the ideal $C^*(G,K)$ is contained in the closed ideal
		$\langle p_e \rangle_{C^*(G)}$ generated by the projection
		$p_e$, and the unital $C^*$-algebra
		$p_e C^*(G)p_e$ decomposes into an orthogonal direct sum
		$p_oC^*(G)p_o \oplus (p_e-p_o)C^*(G)(p_e-p_o)$. Indeed, this decomposition follows from the fact that $p_oC^*(G)p_o$ is an ideal in $p_e C^*(G)p_e$. Since
		$p_e C^*(G)p_e$ is Morita-equivalent to $\langle p_e\rangle_{C^*(G)}$,
		and since $p_e C^*(G)p_e = p_oC^*(G)p_o \oplus (p_e-p_o)C^*(G)(p_e-p_o)$,
		the spectrum $(\widehat{G}_K)_1$ of $C^*(G,K)$, considered
		as a subspace of $\reallywidehat{\langle p_e\rangle_{C^*(G)}}$, is closed.
		But this implies that $[\pi]$ does not belong to the closure of $(\widehat{G}_K)_1$ when	considered as a subspace of $\widehat{G}$ as well.
		
		We have thus excluded the possibility that super-cuspidal and special representations	form points in $\widehat{G}$ that lie in the closure of $(\widehat{G}_K)_1$.	Since these points form the complement of $(\widehat{G}_K)_1$ in $\widehat{G}$, it follows that $(\widehat{G}_K)_1$ is closed in $\widehat{G}$.
	\end{proof}

\subsection{Proofs of Theorem \ref{thm:ideals} and \ref{thm:star_iso}}	
We are now ready to prove Theorem \ref{thm:ideals} and Theorem \ref{thm:star_iso}.
\begin{proof}[Proof of Theorem \ref{thm:ideals}]
Let $G$ be a non-compact, closed subgroup of $\mathrm{Aut}(T)$ that
acts transitively on $V(T)$ and on $\partial T$. Furthermore, suppose that
$G$ has Tits' independence property.

The uniqueness statement in Theorem \ref{thm:ideals} relies on Theorem \ref{thm:groupcstaralgebrastrees}. Let $C^*_\mu(G)$ be a group $C^*$-algebra
such that the dual space $C^*_\mu(G)^*$ is an ideal in $B(G)$. We have to show that $C^*_\mu(G)$ is an $L^{p+}$-group-$C^*$-algebra for a suitable
$p\in[2,\infty]$.

For this purpose, let $o\in V(T)$ be any vertex of $T$ and $K= G_o$.
Theorem \ref{GA:thm:decomposition_grp_alg_auto_grp} implies that $\widehat{G}\setminus \widehat{G}_r\subset (\widehat{G}_K)_1$. As was explained before, the group $G$ is $K$-amenable. From Theorem \ref{CR:thm:KK_eq_abs_grp_alg}, it follows that the canonical quotient map $s\colon C_\mu^*(G)\to C^*_r(G)$ induces
	an isomorphism $s_*\colon K_i(C^*_\mu(G))\to K_i(C^*_r(G)$
	for $i\in \{0,1\}$. By \cite[Lemma 3.3]{siebenand}, this in turn implies that the canonical quotient map $s \vert \colon C^*_\mu(K\backslash G/K)\to C^*_r(K\backslash G/K)$ induces an isomorphism $(s\vert)_*\colon K_i(C^*_\mu(K\backslash G/K))\to K_i(C^*_r(K\backslash G/K))$ for $i\in \{0,1\}$.
	Since $\mu_1$ generates the unital $C^*$-algebra
	$C^*_{\mu}(K\backslash G/K)$, the Gelfand transformation leads to a 
	*-iso\-mor\-phism from $C^*_\mu(K\backslash G/K)$ to
	$C(\sigma_{C^*_\mu(K\backslash G/K)}(\mu_1))$. Let $r\in \sigma_{C^*_\mu(K\backslash G/K)}(\mu_1)$ be
	the largest number in $\sigma_{C^*_\mu(K\backslash G/K)}(\mu_1)$. From Remark \ref{GA:rem:tree_char_vs_spec_of_mu} and the assumption that $G$ acts transitively on $V(T)$, it follows that
	$\sigma_{C^*_\mu(K\backslash G/K)}(\mu_1)\subset[-r,r]$ and that
	$-r$ belongs to $\sigma_{C^*_\mu(K\backslash G/K)}(\mu_1)$ as well.
	
	By Proposition \ref{GA:prp:spectrum_of_mu_in_radial_algs}, there is an
	element $p\in [2,\infty]$ with
	\[\sigma_{C^*_{L^{p+}}(K\backslash G/K)}(\mu_1)=[-r,r].\]
	It therefore suffices to show that $\sigma_{C^*_\mu(K\backslash G/K)}(\mu_1)=[-r,r]$. In order to do this, let $r_2$ be the spectral radius of $\mu_1$ in 
	$C^*_r(K\backslash G/K)$. Then we have $\sigma_{C^*_r(K\backslash G/K)}(\mu_1)
	= [-r_2,r_2]$, and the quotient map $s\vert$ translates
	to the restriction map $\mathrm{res}\colon C(\sigma_{C^*_\mu(K\backslash G/K)}(\mu_1))
	\to C([-r_2,r_2])$. The kernel of $\mathrm{res}$ is equal to
	$C_0(\sigma_{C^*_\mu(K\backslash G/K)}(\mu_1)\setminus [-r_2,r_2] )$.
	From the six-term exact sequence in K-theory and the fact that $\mathrm{res}$
	induces isomorphisms in K-theory, it follows that $C_0(\sigma_{C^*_\mu(K\backslash G/K)}(\mu_1)\setminus [-r_2,r_2] )$ has trivial K-theory. So, if there were an element
	$t \in [-r,r]$ that would not belong to $\sigma_{C^*_\mu(K\backslash G/K)}(\mu_1)$, then $\sigma_{C^*_\mu(K\backslash G/K)}(\mu_1)\setminus [-r_2,r_2]$ would contain a non-empty compact open subset, so that
	\[
		K_0(C_0(\sigma_{C^*_\mu(K\backslash G/K)}(\mu_1)\setminus [-r_2,r_2] ))\neq 0.
		\]
	Hence, we must have that $\sigma_{C^*_\mu(K\backslash G/K)}(\mu_1) = [-r,r]$, which completes the	proof.
\end{proof}
We also prove a version of Theorem \ref{thm:ideals} for an important class of groups that do not satisfy Tits' independence property, namely the groups $\mathrm{SL}(2,\mathbb{Q}_q)$. We refer to explicit results on the representation theory of these groups without recalling them explicitly.
\begin{thm}
	Let $q$ be an odd prime number, and let	$G= \mathrm{SL}(2,\mathbb{Q}_q)$. Let $C^*_{\mu}(G)$ be a group $C^*$-algebra of $G$ whose dual space is a $G$-invariant ideal of $B(G)$. Then there exists a unique element
	$p\in [2,\infty]$ such that
	$C^*_{L^{p+}}(G) = C^*_\mu(G)$.
\end{thm}
\begin{proof}
	The group $G$ acts on a homogeneous tree of degree $q+1$. To be more
	precise, $G$ is a double cover of the projective special linear group
	$\overline{G}=\mathrm{PSL}(2,\mathbb{Q}_q)$, which
	is known to be a non-compact, closed subgroup of the automorphism
	group of the homogeneous tree $T_q$ of degree $q+1$ (see \cite[Appendix 5)]{MR1152801}. It is also known that
	$\overline{G}$ acts transitively on the boundary of $T_q$
	(see \cite[p.~133]{MR1152801}).
	In particular, $\overline{G}$ is a Kunze-Stein group. It also admits Gelfand pairs. Indeed, the group
		\begin{align*}
			K:=\mathrm{SL}(2,\mathbb{Z}_q) = \left\{ \begin{pmatrix}
			a & b \\ 
			c & d
			\end{pmatrix} \in \mathrm{SL}(2,\mathbb{Q}_q)\; \middle| \;
			a,b,c,d\in \mathbb{Z}_q \right\},
		\end{align*}
		where $\mathbb{Z}_q\subset \mathbb{Q}_q$ denotes the ring of integers, is a maximal compact subgroup of $G$, and the image $\overline{K}$ of
		$K$ under the canonical quotient map $G\to \overline{G}$ gives a
		stabiliser group of some vertex $o$ of $T_q$ (see \cite[p. 133]{MR1152801}).
		
		Hence, $(\overline{G},\overline{K})$ is a Gelfand pair, from which it follows that $(G,K)$ is a Gelfand pair as well.
		
		It now follows from the classification of the irreducible unitary group
		representations of $G$ that  $\widehat{G} \setminus \widehat{G}_r \subset (\widehat{G}_K)_1$ (see e.g. \cite{gelcprimefandAutomorphicFunctionsTheory1963} or
		\cite{sallyIntroductionAdicFields1998}). Also, it is straightforward to verify that $(\widehat{\overline{G}}_{\overline{K}})_1 = (\widehat{G}_K)_1$. 
		
		Futhermore, the	restriction of every positive definite, spherical function for $(\mathrm{Aut}(T_q),\mathrm{Aut}(T_q)_o)$
		to $\overline{G}$ is a positive definite, spherical function for $(\overline{G},\overline{K})$.
		
		Combined with the results in \cite[Section 14]{sallyIntroductionAdicFields1998}
		(see also \cite[p.~80]{choucrounAnalyseHarmoniqueGroupes1994}),
		it is straightforward to verify that the kernel of the map
		$q \vert \colon C^*(K\backslash G/K)\to C^*_r(K\backslash G/K)$
		corresponds to
		$\lbrace \varphi_{\gamma_o(s)}\mid s\in [0,\frac{1}{2})\rbrace$.
		The theorem thus follows by arguments analogous to the arguments from the proof of the 
		previous theorem.
\end{proof}
\begin{rmk}
The case of groups $G$ not acting transitively on $V(T)$ requires further investigation. In this case it appears plausible that besides the above mentioned group $C^*$-algebras, there are also other group $C^*$-algebras coming from $G$-invariant ideals of the Fourier-Stieltjes algebra $B(G)$.
\end{rmk}

We end this section with a result (Theorem \ref{thm:star_iso}) indicating the subtle nature of canonical (non)-${}^*$-isomorphism of (exotic) group $C^*$-algebras. Indeed, it shows that even though we may have a ``continuum of exotic group $C^*$-algebras'', in the sense that they are pairwise not canonically ${}^*$-isomorphic, these algebras are still abstractly ${}^*$-isomorphic. 

\begin{proof}[Proof of Theorem \ref{thm:star_iso}]
Let $G$ be a non-compact, closed subgroup of $\mathrm{Aut}(T)$ that acts transitively on $\partial T$ and that satisfies Tits' independence property. Let $o\in V(T)$ be any vertex and $K= G_o$.

	By Theorem \ref{GA:thm:decomposition_grp_alg_auto_grp}, we have
	$C^*(G) = C^*(G)_{\mathrm{disc}} \oplus C^*(G,K)$, where
	the spectrum of $C^*(G)_{\mathrm{disc}}$ consists of $L^2$-representations.
	For every  group $C^*$-algebra $C^*_\mu(G)$, we have a similar
	decomposition $C^*_\mu(G) = C^*_\mu(G)_{\mathrm{disc}}\oplus C^*_\mu(G,K)$
	such that the canonical quotient map
	$q\colon C^*(G) \to C^*_\mu(G)$ transfers
	$C^*(G)_{\mathrm{disc}}$ to $C^*_\mu(G)_{\mathrm{disc}}$
	and $C^*(G,K)$ to $C^*_\mu(G,K)$, i.e.~we have $q(C^*(G)_{\mathrm{disc}})=
	C^*_\mu(G)_{\mathrm{disc}}$ and $q(C^*(G,K))
	=C^*_\mu(G,K)$. Note that $q$ induces a *-iso\-mor\-phism
	from $C^*(G)_{\mathrm{disc}}$ to $C^*_\mu(G)_{\mathrm{disc}}$.
	
	Now let $C^*_\mu(G)$ be a group $C^*$-algebra whose dual space is
	a proper ideal in $B(G)$. Furthermore, let $s\colon C^*_\mu(G) \to C^*_r(G)$ be the canonical quotient map. A similar argument as above shows
	that $s$ induces a *-isomorphism from $C^*_\mu(G)_{\mathrm{disc}}$
	to $C^*_r(G)_{\mathrm{disc}}$. Furthermore,
	$s$ restricts to a surjective *-homomorphism $\colon C^*_\mu(G,K)\to C^*_r(G,K)$.
	
	It remains to show that $C^*_\mu(G,K)$ and $C^*_r(G,K)$ are *-isomorphic. By Remark \ref{GA:rem:GP_spherical_ideal_morita_eq}, the algebra $C^*_\mu(G,K)$
	is Morita-equivalent to $C_\mu^*(K\backslash G/K)$ and 
	$C^*_r(G,K)$ is Morita-equivalent to $C^*_r(K\backslash G/K)$.
	In particular, it follows that $C^*_\mu(G,K)$ and $C^*_r( G,K)$ are	separable continuous trace $C^*$-algebras (see e.g. \cite[Corollary IV.1.4.20]{MR2188261}).
	Since $C_\mu^*(G)^*$ is an ideal of $B(G)$, every
	irreducible unitary group representation of $G$ that extends to $C^*_\mu(G)$ is
	infinite dimensional.
	
	From \cite[Corollary IV.1.7.22]{MR2188261}, it follows that
	$C^*_\mu( G,K)$ and $C^*_r( G,K)$ are stable $C^*$-algebras.
	An argument similar to that from the proof of the previous theorems shows
	that the Gelfand spaces $\Delta(C_\mu^*(K\backslash G/K))$ and 
	$\Delta(C_r^*(K\backslash G/K))$ are homeomorphic (both are perfect,
	compact intervals).
	From the Dixmier-Douady classification of continuous trace $C^*$-algebras, it follows that $C^*_\mu( G,K)$ and $C^*_r( G,K)$ are Morita-equivalent (see \cite[Theorem IV.1.7.11]{MR2188261}).
	This implies that $C^*_\mu( G,K)$ and $C^*_r(G,K)$
	are stably *-isomorphic, and since $C^*_\mu( G,K)$ and $C^*_r( G,K)$ are stable, they are actually *-iso\-mor\-phic.
\end{proof}


\begin{thebibliography}{MMM88}

\bibitem[Ama03]{amann}
O.\'E{}.~Amann, \emph{Groups of tree-automorphisms and their unitary representations}, Ph.D.~Thesis, ETH Zurich, 2003.
	
\bibitem[AB+19]{antoninietal}
P.~Antonini, A.~Buss, A.~Engel and T.~Siebenand, \emph{Strong Novikov conjecture for low degree cohomology and exotic group $C^*$-algebras},
preprint (2019), arXiv:1905.07730.

\bibitem[BGW16]{MR3514939}
P.~Baum, E.~Guentner, R.~Willett, \emph{Expanders, exact crossed products, and the Baum-Connes conjecture}, Ann.~K-Theory \textbf{1} (2016), 155--208. 

\bibitem[BHV08]{bekkaKazhdanProperty2008a}
B. Bekka, P. {de la Harpe} and A. Valette, \emph{Kazhdan's property
  ({{T}})}, {Cambridge University
  Press, Cambridge}, 2008. 
  
\bibitem[Bla06]{MR2188261}
B.~Blackadar, \emph{Operator Algebras}, Springer-Verlag, Berlin, 2006.
  
\bibitem[BG13]{MR3138486}
N.P.~Brown and E.P.~Guentner, \emph{New {$C^*$}-completions of discrete groups and related spaces}, Bull. Lond. Math. Soc. \textbf{45} (2013), 1181--1193.

\bibitem[BT72]{MR0327923}
F.~Bruhat and J.~Tits, \emph{Groupes r\'eductifs sur un corps local}, Publ. Math. IHES \textbf{41} (1972), 5--251. 

\bibitem[BM00]{MR1839488}
M.~Burger and S.~Mozes, \emph{Groups acting on trees: from local to global structure}, Publ. Math. IHES \textbf{92} (2000), 113--150. 

\bibitem[BEW17]{MR3837592}
A.~Buss, S.~Echterhoff and R,~Willett, \emph{Exotic crossed products}, Operator algebras and applications--the {A}bel {S}ymposium 2015,
Abel Symp., Vol.~12, Springer, 2017, 67--114.

\bibitem[BEW18]{MR3824785}
\bysame, \emph{Exotic crossed products and the {B}aum-{C}onnes conjecture},
J.~Reine Angew. Math. \textbf{740} (2018), 111--159.

\bibitem[Cho94]{choucrounAnalyseHarmoniqueGroupes1994}
F.~M. Choucroun, \emph{Analyse harmonique des groupes d'automorphismes
  d'arbres de {{Bruhat}}-{{Tits}}}, M\'em. Soc. Math. France (N.S.) (1994),
  no.~58.

\bibitem[Cun83]{MR0716254}
J.~Cuntz, \emph{K-theoretic amenability for discrete groups},
J.~Reine Angew.~Math. \textbf{344} (1983), 180--195. 

\bibitem[Dix77]{MR0458185}
J.~Dixmier, \emph{$C^*$-Algebras}, North-Holland Publishing Co., Amsterdam - New York - Oxford, 1977. 

\bibitem[DM76]{dufloRegularRepresentationNonunimodular1976}
M.~Duflo and C.~Moore, \emph{On the regular representation of a nonunimodular locally compact group}, J. Funct. Anal. \textbf{21} (1976), 209--243.

\bibitem[FTN91]{MR1152801}
A.~Fig\`a-Talamanca and C.~Nebbia,
\emph{Harmonic Analysis and Representation Theory for Groups Acting on Homogeneous Trees}, Cambridge University Press, Cambridge, 1991. 

\bibitem[FP83]{MR0710827}
A.~Fig\`a-Talamanca and M.A.~Picardello, \emph{Harmonic Analysis on Free Groups}, Marcel Dekker, Inc., New York, 1983.

\bibitem[GGT18]{garridoAutomorphismGroupsTrees2018}
A. Garrido, Y. Glasner and S. Tornier, \emph{Automorphism groups
  of trees: Generalities and prescribed local actions}, in: New Directions in
  Locally Compact Groups, {Cambridge Univ. Press, Cambridge}, 2018, pp.~92--116. 
  
\bibitem[GPS63]{gelcprimefandAutomorphicFunctionsTheory1963}
I.M. {Gel'fand} and I.I. {Pyatetskii-Shapiro}, \emph{Automorphic functions and the theory of representations}, Trudy Moskov. Mat. Ob\v{s}\v{c}. \textbf{12} (1963),
  389--412.

\bibitem[JV84]{MR0757995}
P.~Julg and A.~Valette, \emph{K-theoretic amenability for $\mathrm{SL}_2(\mathbb{Q}_p)$, and the action on the associated tree},
J.~Funct. Anal. \textbf{58} (1984), 194--215. 

\bibitem[KLQ13]{MR3141810}
S.~Kaliszewski, M.~Landstad and J.~Quigg, \emph{Exotic group $C^*$-algebras in noncommutative duality}, New York J.~Math.~\textbf{19} (2013), 689--711.

\bibitem[dLS19]{delaatsiebenand1}
T.~de Laat and T.~Siebenand, \emph{Exotic group $C^*$-algebras of simple Lie groups with real rank one}, to appear in Ann. Inst. Fourier (Grenoble), arXiv:1912.02128.

\bibitem[Neb88]{MR0936361}
C.~Nebbia, \emph{Groups of isometries of a tree and the Kunze-Stein phenomenon}, Pacific J.~Math.~\textbf{133} (1988), 141--149. 

\bibitem[Oka14]{MR3238088}
R.~Okayasu, \emph{Free group {$C^*$}-algebras associated with {$\ell_p$}},
Internat.~J.~Math.~\textbf{25} (2014), 1450065, 12.~pp.

\bibitem[Ol'77]{MR0578650}
G.I.~Ol'\v{s}anski\u{\i}, \emph{Classification of the irreducible representations of the automorphism groups of Bruhat-Tits trees}, Funkcional.~Anal.~i Prilo\v{z}en \textbf{11} (1977), 32--42, 96. 

\bibitem[Sal98]{sallyIntroductionAdicFields1998}
P.J.~Sally, Jr., \emph{An introduction to $p$-adic fields, harmonic
  analysis and the representation theory of $\mathrm{SL}_2$},
  Lett. Math. Phys. \textbf{46} (1998), 1--47.

\bibitem[SW18]{sameiwiersma2}
E.~Samei and M.~Wiersma, \emph{Exotic {$C^{*}$}-algebras of geometric groups}, preprint (2018), arXiv:1809.07007.

\bibitem[Sie20a]{siebenanddissertation}
T.~Siebenand, \emph{Exotic group $C^*$-algebras and crossed products}, Ph.D.~Thesis, Westf\"alische Wilhelms-Universit\"at M\"unster, 2020.

\bibitem[Sie20b]{siebenand}
T.~Siebenand, \emph{On the ideal structure of the Fourier-Stieltjes algebra of certain groups}, preprint (2020), arXiv:2005.01772.

\bibitem[Tit70]{MR0299534}
J.~Tits, \emph{Sur le groupe des automorphismes d'un arbre}, Essays on topology and related topics, pp.~188--211. Springer, New York, 1970. 

\bibitem[Tu99]{MR1703305}
J.-L.~Tu, \emph{La conjecture de Baum-Connes pour les feuilletages moyennables},
K-Theory \textbf{17} (1999), 215--264. 

\bibitem[Wie15]{MR3418075}
M.~Wiersma, \emph{$L^p$-Fourier and Fourier-Stieltjes algebras for locally compact groups}, J.~Funct.~Anal.~\textbf{269} (2015), 3928--3951.

\bibitem[Wie16]{MR3705441}
\bysame, \emph{Constructions of exotic group {$C^*$}-algebras},
Illinois J.~Math.~\textbf{60} (2016), 655--667.

\bibitem[Wol07]{MR2328043}
J.A.~Wolf, \emph{Harmonic Analysis on Commutative Spaces}, Amer. Math. Soc., Providence, RI, 2007.

\end{thebibliography}
\end{document}